\documentclass[a4paper]{article}
\AtBeginDocument{%
  \providecommand\BibTeX{{%
    \normalfont B\kern-0.5em{\scshape i\kern-0.25em b}\kern-0.8em\TeX}}}

\usepackage{mathtools}

\usepackage[english]{babel}
\usepackage{amsfonts}
\usepackage{amsthm}
\usepackage{amsmath}
\usepackage{amssymb}
\usepackage{fouriernc}
\usepackage{color}
\usepackage{tikz}
\usepackage{enumitem}
\usepackage{hyperref}
\usepackage{caption}

\def\eatspace#1{#1}
\def\step#1#2{\par\kern1pt\dimen44=#2em\advance\dimen44 1.67em\hangindent\dimen44\hangafter=1\noindent\rlap{\small#1}\kern\dimen44\relax\eatspace}

\def\<#1>{\langle#1\rangle}

\theoremstyle{definition}

\makeatletter
\def\testb#1{\testb@i#1,,\@nil}%
\def\testb@i#1,#2,#3\@nil{%
  \draw[->] (O) --++(#1);
  \ifx\relax#2\relax\else\testb@i#2,#3\@nil\fi}
\makeatother

\usepackage[utf8]{inputenc}
\usepackage{mathtools}
\usepackage{amsfonts}
\usepackage{amsthm}
\usepackage{amsmath}
\usepackage{amssymb}
\usepackage{tikz}
\usepackage{hyperref}
\usepackage{enumitem}

\newtheorem{Proposition}{Proposition}
\newtheorem{Lemma}{Lemma}

\newtheorem{Definition}{Definition}

\newcommand{\uproman}[1]{\uppercase\expandafter{\romannumeral#1}}

\usepackage{pgfplots}\pgfplotsset{compat=1.13}
\usetikzlibrary{arrows,positioning,calc,patterns}

\def\eatspace#1{#1}
\def\step#1#2{\par\kern1pt\dimen44=#2em\advance\dimen44 1.67em\hangindent\dimen44\hangafter=1\noindent\rlap{\small#1}\kern\dimen44\relax\eatspace}

\def\<#1>{\langle#1\rangle}

\def\val{\operatorname{val}}

\AtBeginDocument{%
  \providecommand\BibTeX{{%
    Bib\TeX}}}

\begin{document}

\title{
$x(1-t(x + x^{-1})) F(x;t) = x - t F(0;t)$

%\large 
%\begin{center}
%A survey on the enumeration of lattice walks
%\end{center}

}

\date{\vspace{-1.75 cm}}

%\author{Manfred Buchacher}
%\affiliation{
%%  \department{RICAM}
%  \institution{Austrian Academy of Sciences}
%  \city{Linz}
%  \country{Austria}
%}
%\email{manfredi.buchacher@gmail.com}

%%f

\maketitle

\begin{center}
\begin{tabular}{@{}c@{}}
    Manfred Buchacher \\
    Johannes Kepler Universit{\"a}t Linz\\
    \normalsize manfredi.buchacher@gmail.com
  \end{tabular}%
  \end{center}
\bigskip

%Das naechste ist, dasz ein erster Absatz fehlt, der den Kontext
%des  Artikels einfuehrt. Also, der davon erzaehlt, dasz es in den
%letzten Jahren intensive Forschungen ueber walks in Quadranten
%(und andeeren Regionen) gegeben hat, mit der Frage, ihre
%Anzahl zu bestimmen, was eine explizite Formel sein koennte, oder
%auch die Bestimmung der entsprechenden erzeugenden Funktion, deren
%Natur auch insbesondere untersucht wurde. Ueber die Zeit wurden
%dazu einige Methoden und Herangehensweisen entwickelt, die bedeutend
%sind. Du moechtest darueber einen Ueberblick bieten und tust dieses,
%indem Du sie an Hand des einfachsten nichttrivialen Beispiels
%demonstrierst.
%
%The enumeration of lattice walks has received considerable attention in the last decades. Initiated by~\cite{MISHNA2009460} and~\cite{smallSteps}, the focus has mainly been on walks restricted to convex cones. Questions about exact and asymptotic number of walks have been mainly approached via their associated generating functions and the algebraic properties they satisfy.  
%\bigskip

The enumeration of lattice walks restricted to cones has received considerable attention since Bousquet-M\'{e}lou initiated their systematic study in the early 2000's. 
%since Bousquet-M\'{e}lou and Mishna initiated their systematic study~\cite{smallSteps}. 
The question how to count them and how to find a ``nice'' formula for their number 
%- may it be exact or asymptotic - 
can be approached by considering their generating functions and investigating the functional equations they satisfy. Numerous people have done so by now and there are meanwhile many methods these equations can be studied with.  
%Many people have worked on it, and there are meanwhile numerous methods and techniques to investigate their generating functions and the functional equations they satisfy. 
The methods are diverse. They are algebraic and / or analytic. They may rely on elementary power series algebra~\cite{smallSteps, bousquet2016elementary, Tutte, Mishna1} and constructive and computer algebra~\cite{BostanKauers, bostan2017hypergeometric}, or involve complex analysis and boundary value problems~\cite{fayolle1999random, raschel2012counting, Tutte}, Galois theory of difference equations~\cite{dreyfus2018nature, dreyfus2020walks, hardouin2021differentially}, analytic combinatorics~\cite{flajolet2009analytic, PemantleWilsonMelczer2024}, or probability theory~\cite{fayolle1999random, denisov2015random}.
\bigskip

The purpose of these notes is to introduce some of the problems the enumeration of lattice walks is dedicated to and familiarize 
%in a rather playful way 
with some of the arguments they can be addressed with. We discuss the enumeration of lattice walks, their generating functions, and the functional equations they satisfy. We focus on algebraic methods for manipulating and solving these equations. Elementary power series algebra plays a prominent role, computer algebra too, but we repeatedly digress and present ideas and methods of different kind whenever it is appropriate. The exposition is organized around the most simple yet non-trivial problem: the enumeration of simple walks on the half-line.
%: the enumeration of simple walks on the half-line. 
The intention is to illustrate different techniques without getting technical. 
%In view of the simplicity of the problem, it has the paradoxical and sometimes amusing -- matter of taste! -- consequence that things appear relatively complicated.
\bigskip

The purpose of the notes is not to give an exhaustive overview of the literature on the enumeration of lattice walks. The subject is much too vast. We limit ourselves to the study of functional equations and the 
algebraic and related methods that could have originated from that. The pointers to the literature we give are therefore limited to articles that are relevant in this context. For a broader and more detailed exposition of the subject we refer to the surveys by Krattenthaler~\cite{Krattenthaler} and Humphreys~\cite{humphreys2010history}.     
\bigskip

%At the time of writing it was not clear at all that these notes would grow to such a size. It was a surprise to see how things worked out, and it was fun to work them out. We hope the reader can feel some of the joy this was done with.
%\bigskip  

The notes are organized in 14 sections. Section~\ref{sec:genF} provides the basic definitions and introduces lattice walks, their generating functions, and the functional equations they satisfy. In Sections~\ref{sec:kernelMethod},~\ref{sec:wienerHopf},~\ref{sec:orbitSum},~\ref{sec:compInv} and ~\ref{sec:invMethod} we solve such equations using methods that rely on elementary power series algebra only. We discuss the kernel method, Wiener-Hopf factorization, the orbit-sum method, compositional inverses and Lagrange inversion as well as the invariant method. These methods give rise to expressions for the generating functions that beg for combinatorial explanations. Such explanations are given in Sections~\ref{sec:contFrac},~\ref{sec:reflecPrinciple},~\ref{sec:combFact} and~\ref{sec:cycleLemma} and relate them to continued fractions, the reflection principle, a combinatorial factorization of lattice walks and the cycle lemma. In Section~\ref{sec:guessProve} and Section~\ref{sec:diffAlg} we illustrate the paradigm of guess and prove and discuss the usefulness of the notion of D-finiteness. They connect enumerative combinatorics with computer and differential algebra. Section~\ref{sec:asympMethod} discusses asymptotics, generalized series solutions of linear recurrences, and how they relate to each other. A discussion of the references much of this text relies on is postponed to Section~\ref{sec:references}. 
%The methods illustrated here were invented there, or appeared there in much more generality. 
Some methods for the enumeration of lattice walks are not addressed at all though they do fall into the framework of this text. We only mention the half-orbit sum method, the obstinate kernel method and the iterated kernel method here. Pointers to the relevant literature are given in Section~\ref{sec:references}.    

%Weiters solltest Du auch an frueher Stelle sagen, dasz Du Dich
%dafuer entschieden hast, die Diskussion zunaechst ohne ausfuehrliche
%Referenzen zu fuehren, dasz Du aber dann im letzten Abschnitt
%ausfuehrlich Geschichte und weiterfuehrende Literatur ausfuehrst

\setcounter{section}{-1}

\section{Notation} 
We denote by $\mathbb{N}$, $\mathbb{Z}$, $\mathbb{Q}$ and $\mathbb{C}$ the non-negative integers, the integers and the rational and complex numbers, respectively. We denote by $x$ and $t$ variables. We write $\mathbb{Q}[x]$ for the ring of polynomials in $x$ over $\mathbb{Q}$, and we denote by $\mathbb{Q}(x)$ its quotient field. Defining $\bar{x}:=x^{-1}$, for notational convenience, we denote by $\mathbb{Q}[x,\bar{x}]$ the ring of Laurent polynomials in $x$. Given a ring $A$ such as $\mathbb{Q}$, $\mathbb{Q}[x]$, $\mathbb{Q}[\bar{x}]$ or $\mathbb{Q}[x, \bar{x}]$, we denote by $A[[t]]$ and $A((t))$ the rings of power and Laurent series in $t$ over $A$.      
%The ring of power series in $t$ whose coefficients are Laurent polynomials in $x$ over $\mathbb{Q}$ is denoted by $\mathbb{Q}[x,\bar{x}][[t]]$. Similarly, we write $\mathbb{Q}[x][[t]]$ and $\mathbb{Q}[\bar{x}][[t]]$ for the subrings of series whose coefficients are polynomials in $x$ and $\bar{x}$, respectively, and $\mathbb{Q}[[t]]$ for those series which do not involve $x$. 
Given a series $F(x;t)$ in $\mathbb{Q}[x,\bar{x}][[t]]$ and an integer $n$, we denote by $[t^n] F(x;t)$ the coefficient of $t^n$, and we define by
\begin{equation*}
[x^{\geq }] F(x;t) := \sum_{n\geq 0} x^n [x^n]F(x;t)
\end{equation*}
its \emph{non-negative part}.

\section{Generating functions and functional equations}\label{sec:genF}

\begin{Definition}
A \emph{lattice walk} is a sequence $P_0, \dots, P_n$ of points of $\mathbb{Z}^d$, $d\in\mathbb{N}$. The points $P_0$ and $P_n$ are its starting and end point, respectively, the consecutive differences $P_{i+1}-P_i$ are its steps, and $n$ is its length. Fixing $R$, $S\subseteq \mathbb{Z}^d$, we denote by $f(Q;n)$ the number of walks in $R$ that start at the origin, consist of $n$ steps all of which are taken from $S$, and end in $Q\in\mathbb{Z}^d$. The series 
\begin{equation*}
F(x;t) := \sum_{n\geq 0} \left( \sum_{Q\in R} f(Q;n) x^Q \right) t^n 
\end{equation*}
is called the \emph{generating function} of the sequence $(f(Q;n))$.
\end{Definition}

\begin{figure}[ht]
\hspace{-1.5cm}
  \centering
  \begin{center}
  \begin{tikzpicture}[scale=.4]
   
    \coordinate (Origin)   at (-2,-1);
    \coordinate (XAxisMin) at (-3,-1);
    \coordinate (XAxisMax) at (21.5,-1);
    \coordinate (YAxisMin) at (-2,-2);
    \coordinate (YAxisMax) at (-2,3.5);
    \draw [thin, gray,-latex] (XAxisMin) -- (XAxisMax);% Draw x axis
    \draw [thin, gray,-latex] (YAxisMin) -- (YAxisMax);% Draw y axis
 
    \coordinate (s1) at (1,1);
    \coordinate (s2) at (1,-1);

    \draw[style=help lines,dashed] (-2,-1) grid[step=1cm] (20,2);

    %draws the steps of the walk

%    \draw [thick,-latex,teal] (Origin)
%        -- ($(Origin)+(s1)$) ;
%
%    \draw [thick,-latex,teal] ($(Origin)+(s1)$)
%        -- ($(Origin)+(s1)+(s2)$) node [above right] {};
%
%    \draw [thick,-latex,teal] ($(Origin)+(s1)+(s2)$)
%        -- ($(Origin)+(s1)+(s2)+(s1)$) node [above right] {};
%
%    \draw [thick,-latex,teal] ($(Origin)+(s1)+(s2)+(s1)$)
%        -- ($(Origin)+(s1)+(s2)+(s1)+(s1)$) node [above right] {};
%
%    \draw [thick,-latex,teal] ($(Origin)+(s1)+(s2)+(s1)+(s1)$)
%        -- ($(Origin)+(s1)+(s2)+(s1)+(s1)+(s2)$) node [above right] {};
%
%    \draw [thick,-latex,teal] ($(Origin)+(s1)+(s2)+(s1)+(s1)+(s2)$)
%        -- ($(Origin)+(s1)+(s2)+(s1)+(s1)+(s2)+(s2)$) node [above right] {};
%        
%        \draw [thick,-latex,teal] ($(Origin)+(s1)+(s2)+(s1)+(s1)+(s2)+(s2)$)
%        -- ($(Origin)+(s1)+(s2)+(s1)+(s1)+(s2)+(s2)+(s2)$) node [above right] {};
        
           \draw [thick,-latex,teal] ($(Origin)$)
        -- ($(Origin)+(s1)$) node [above right] {};
        
               \draw [thick,-latex,teal] ($(Origin)+(s1)$)
        -- ($(Origin)+(s1)+(s1)$) node [above right] {};

               \draw [thick,-latex,teal] ($(Origin)+(s1)+(s1)$)
        -- ($(Origin)+(s1)+(s1)+(s2)$) node [above right] {};
        
                \draw [thick,-latex,teal] ($(Origin)+(s1)+(s1)+(s2)$)
        -- ($(Origin)+(s1)+(s1)+(s2)+(s1)$) node [above right] {};
        
                 \draw [thick,-latex,teal] ($(Origin)+(s1)+(s1)+(s2)+(s1)$)
        -- ($(Origin)+(s1)+(s1)+(s2)+(s1)+(s2)$) node [above right] {};
        
        \draw [thick,-latex,teal] ($(Origin)+(s1)+(s1)+(s2)+(s1)+(s2)$)
        -- ($(Origin)+(s1)+(s1)+(s2)+(s1)+(s2)+(s1)$) node [above right] {};
        
           \draw [thick,-latex,teal] ($(Origin)+(s1)+(s1)+(s2)+(s1)+(s2)+(s1)$)
        -- ($(Origin)+(s1)+(s1)+(s2)+(s1)+(s2)+(s1)+(s1)$) node [above right] {};
        
            \draw [thick,-latex,teal] ($(Origin)+(s1)+(s1)+(s2)+(s1)+(s2)+(s1)+(s1)$)
        -- ($(Origin)+(s1)+(s1)+(s2)+(s1)+(s2)+(s1)+(s1)+(s2)$) node [above right] {};
   
     \draw [thick,-latex,teal] ($(Origin)+(s1)+(s1)+(s2)+(s1)+(s2)+(s1)+(s1)+(s2)$)
        -- ($(Origin)+(s1)+(s1)+(s2)+(s1)+(s2)+(s1)+(s1)+(s2)+(s2)$) node [above right] {};
        
         \draw [thick,-latex,teal] ($(Origin)+(s1)+(s1)+(s2)+(s1)+(s2)+(s1)+(s1)+(s2)+(s2)$)
        -- ($(Origin)+(s1)+(s1)+(s2)+(s1)+(s2)+(s1)+(s1)+(s2)+(s2)+(s2)$) node [above right] {};
        
             \draw [thick,-latex,teal] ($(Origin)+(s1)+(s1)+(s2)+(s1)+(s2)+(s1)+(s1)+(s2)+(s2)+(s2)$)
        -- ($(Origin)+(s1)+(s1)+(s2)+(s1)+(s2)+(s1)+(s1)+(s2)+(s2)+(s2)+(s1)$) node [above right] {};
        
         \draw [thick,-latex,teal] ($(Origin)+(s1)+(s1)+(s2)+(s1)+(s2)+(s1)+(s1)+(s2)+(s2)+(s2)+(s1)$)
        -- ($(Origin)+(s1)+(s1)+(s2)+(s1)+(s2)+(s1)+(s1)+(s2)+(s2)+(s2)+(s1)+(s1)$) node [above right] {};
        
         \draw [thick,-latex,teal] ($(Origin)+(s1)+(s1)+(s2)+(s1)+(s2)+(s1)+(s1)+(s2)+(s2)+(s2)+(s1)+(s1)$)
        -- ($(Origin)+(s1)+(s1)+(s2)+(s1)+(s2)+(s1)+(s1)+(s2)+(s2)+(s2)+(s1)+(s1)+(s2)$) node [above right] {};

 \draw [thick,-latex,teal] ($(Origin)+(s1)+(s1)+(s2)+(s1)+(s2)+(s1)+(s1)+(s2)+(s2)+(s2)+(s1)+(s1)+(s2)$)
        -- ($(Origin)+(s1)+(s1)+(s2)+(s1)+(s2)+(s1)+(s1)+(s2)+(s2)+(s2)+(s1)+(s1)+(s2)+(s1)$) node [above right] {};
        
         \draw [thick,-latex,teal] ($(Origin)+(s1)+(s1)+(s2)+(s1)+(s2)+(s1)+(s1)+(s2)+(s2)+(s2)+(s1)+(s1)+(s2)+(s1)$)
        -- ($(Origin)+(s1)+(s1)+(s2)+(s1)+(s2)+(s1)+(s1)+(s2)+(s2)+(s2)+(s1)+(s1)+(s2)+(s1)+(s2)$) node [above right] {};

    \draw [thick,-latex,teal] ($(Origin)+(s1)+(s1)+(s2)+(s1)+(s2)+(s1)+(s1)+(s2)+(s2)+(s2)+(s1)+(s1)+(s2)+(s1)+(s2)$)
        -- ($(Origin)+(s1)+(s1)+(s2)+(s1)+(s2)+(s1)+(s1)+(s2)+(s2)+(s2)+(s1)+(s1)+(s2)+(s1)+(s2)+(s1)$) node [above right] {};
        
          \draw [thick,-latex,teal] ($(Origin)+(s1)+(s1)+(s2)+(s1)+(s2)+(s1)+(s1)+(s2)+(s2)+(s2)+(s1)+(s1)+(s2)+(s1)+(s2)+(s1)$)
        -- ($(Origin)+(s1)+(s1)+(s2)+(s1)+(s2)+(s1)+(s1)+(s2)+(s2)+(s2)+(s1)+(s1)+(s2)+(s1)+(s2)+(s1)+(s1)$) node [above right] {};
        
            \draw [thick,-latex,teal] ($(Origin)+(s1)+(s1)+(s2)+(s1)+(s2)+(s1)+(s1)+(s2)+(s2)+(s2)+(s1)+(s1)+(s2)+(s1)+(s2)+(s1)+(s1)$)
        -- ($(Origin)+(s1)+(s1)+(s2)+(s1)+(s2)+(s1)+(s1)+(s2)+(s2)+(s2)+(s1)+(s1)+(s2)+(s1)+(s2)+(s1)+(s1)+(s2)$) node [above right] {};

    \draw [thick,-latex,teal] ($(Origin)+(s1)+(s1)+(s2)+(s1)+(s2)+(s1)+(s1)+(s2)+(s2)+(s2)+(s1)+(s1)+(s2)+(s1)+(s2)+(s1)+(s1)+(s2)$)
        -- ($(Origin)+(s1)+(s1)+(s2)+(s1)+(s2)+(s1)+(s1)+(s2)+(s2)+(s2)+(s1)+(s1)+(s2)+(s1)+(s2)+(s1)+(s1)+(s2)+(s2)$) node [above right] {};
        
         \draw [thick,-latex,teal] ($(Origin)+(s1)+(s1)+(s2)+(s1)+(s2)+(s1)+(s1)+(s2)+(s2)+(s2)+(s1)+(s1)+(s2)+(s1)+(s2)+(s1)+(s1)+(s2) +(s2)$)
        -- ($(Origin)+(s1)+(s1)+(s2)+(s1)+(s2)+(s1)+(s1)+(s2)+(s2)+(s2)+(s1)+(s1)+(s2)+(s1)+(s2)+(s1)+(s1)+(s2)+(s2) + (s2)$) node [above right] {};
        
         \draw [thick,-latex,teal] ($(Origin)+(s1)+(s1)+(s2)+(s1)+(s2)+(s1)+(s1)+(s2)+(s2)+(s2)+(s1)+(s1)+(s2)+(s1)+(s2)+(s1)+(s1)+(s2) +(s2)+(s2)$)
        -- ($(Origin)+(s1)+(s1)+(s2)+(s1)+(s2)+(s1)+(s1)+(s2)+(s2)+(s2)+(s1)+(s1)+(s2)+(s1)+(s2)+(s1)+(s1)+(s2)+(s2) + (s2)+(s1)$) node [above right] {};
        
               \draw [thick,-latex,teal] ($(Origin)+(s1)+(s1)+(s2)+(s1)+(s2)+(s1)+(s1)+(s2)+(s2)+(s2)+(s1)+(s1)+(s2)+(s1)+(s2)+(s1)+(s1)+(s2) +(s2)+(s2)+(s1)$)
        -- ($(Origin)+(s1)+(s1)+(s2)+(s1)+(s2)+(s1)+(s1)+(s2)+(s2)+(s2)+(s1)+(s1)+(s2)+(s1)+(s2)+(s1)+(s1)+(s2)+(s2) + (s2)+(s1)+(s2)$) node [above right] {};

  \end{tikzpicture}
  \end{center}
  \caption{A graphical representation of a lattice walk from $(0,0)$ to $(22,0)$ whose steps are either $(1,1)$ or $(1,-1)$.}
 \end{figure}
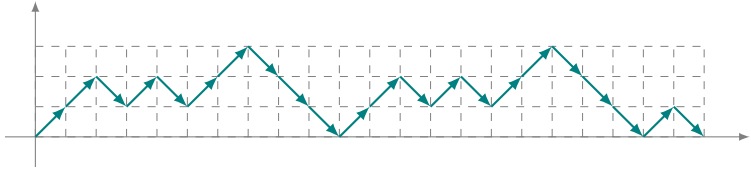

Enumerative combinatorics asks for the value of $f(Q;n)$, whether there is a ``nice'' formula for it in terms of $Q$ and $n$, and what its asymptotic behavior is as $n$ tends to infinity. A first step in answering these questions is to consider the associated generating function and to locate it in the hierarchy of rational, algebraic and D-finite functions -- see Definition~\ref{def:alg} and Definition~\ref{def:dFinite} below. %TODO: give definitions? 
As soon as a representation of the generating function in terms of an algebraic or a differential equation is known, there are classes of algorithms that can be applied to derive information about its coefficient sequence~\cite{kauers2023d}. 
%TODO: gives references? 
The starting point of the investigation is a functional equation for the generating function. In general this equation can be complicated, but sometimes 
%for simple domains and step sets
%, such as $R=\mathbb{N}^d$ and $S\subseteq \{-1,0,1\}^d$ for instance, 
it is easily written down.
\bigskip

In the following, and for the rest of the paper, we restrict our discussion to simple walks on the half-line. Let $F(x;t)$ be the generating function of walks in $\mathbb{N}$ which start at $0$ and whose steps are either $1$ or $-1$. A lattice walk of positive length that ends at $i$ is either a walk to $i-1$ followed by a right step, or a walk to $i+1$ followed by a left step.
%-- any lattice walk is either of length $0$, or some walk followed by $1$, or by $-1$ unless its endpoint is $0$ -- 
Hence,
\begin{equation}\label{eq:rec}
f(i;n) = f(i-1;n-1) + f(i+1;n-1), \quad n \geq 1.  
\end{equation}
Furthermore, there is only one walk of length $0$, and there are no walks that end on $i < 0$. Thus, 
%with the initial and boundary condition, respectively,
\begin{equation}
f(i;0) = \delta_{i,0} \quad \text{and} \quad f(i;n) = 0 \quad \text{if} \quad  i<0.
\end{equation}
Multiplying both sides of equation~\eqref{eq:rec} by $x^it^n$ and summing over all $i \geq 0$ and $n \geq 1$, it easily follows that 
%we easily derive the following functional equation 
\begin{equation}\label{eq:kernelEq}
F(x;t) = 1 + t(x+\bar{x}) F(x;t) - t\bar{x} F(0;t).
\end{equation} 
%for its generating function. 
This equation uniquely determines $F(x;t)$ as it implies 
%. It implies the following recurrence relation
\begin{align*}
[t^0] F(x;t) &= 1,\\
[t^{n+1}] F(x;t) &= (x+\bar{x}) [t^n] F(x;t) - \bar{x} [t^n] F(0;t), \quad n\geq 0,
\end{align*}
for the coefficient sequence $([t^n]F(x;t))$. The first terms of $F(x;t)$ are easily computed based thereon. We have  
\begin{equation*}
F(x;t) = 1 + x t + (1 + x^2) t^2 + (2 x + x^3) t^3 + (2 + 3 x^2 + x^4) t^4 + O(t^5)
\end{equation*}
from which we can read off some of the values of $f(i;n)$. 
\bigskip

The existence and uniqueness of a solution of equation~\eqref{eq:kernelEq} in $\mathbb{Q}[x][[t]]$ also follows from Banach's fixed point theorem. For any $F$, $G\in\mathbb{Q}[x][[t]]$, define their distance from each other by
\begin{equation*}
\mathrm{d}(F,G) := 2^{-\mathrm{val}(F-G)},
\end{equation*}
where $\mathrm{val}$ denotes the valuation 
\begin{equation*}
\mathrm{val}(F) := \min\{n:[t^n]F\neq 0\}.
\end{equation*}
It turns $\mathbb{Q}[x][[t]]$ into a complete metric space on which the operator
\begin{equation*}
F(x;t) \mapsto 1 + t(x+\bar{x})F(x;t) - t\bar{x}F(0;t) 
\end{equation*}
acts as a contraction. So there is a unique solution in $\mathbb{Q}[x][[t]]$, and the sequence of iterates of any element in $\mathbb{Q}[x][[t]]$ converges to it. 
%TODO: explain notation, further diagrams? 
\bigskip

To find a closed form expression for $F(x;t)$ it is convenient to rewrite equation~\eqref{eq:kernelEq} in the form
\begin{equation}\label{eq:kEq}
x(1-t(x + \bar{x})) F(x;t) = x - t F(0;t).
\end{equation}
This equation is referred to as the \emph{kernel equation}. The coefficient of $F(x;t)$ is the \emph{kernel polynomial}. The latter plays, as the name indicates, an important role in the different solution strategies subsumed under what is known as the \emph{kernel method}. 

\section{Classical kernel method}\label{sec:kernelMethod}

%To solve equation~\eqref{eq:kEq}, one important observation is that it involves two unknowns, $F(x;t)$ and $F(0;t)$, one of which does not depend on $x$. 

%There are different approaches to solve equation~\eqref{eq:kEq}. 
%We focus on algebraic rather than on analytic methods and put an emphasis on showing how the kernel equation can be solved using only the operations
%\begin{equation*}
%+, \quad \cdot, \quad \circ \quad \text{ and } \quad [x^>].
%\end{equation*}
%That is, the addition, multiplication and composition of series, and the operation of discarding all terms of a series which involve non-positive powers in $x$. 

The kernel equation involves two unknown functions one of which involves only one of  the variables. The classical kernel method is based on the idea of eliminating one of the unknowns by coupling $x$ and $t$ without altering the other unknown. 
%One way to solve equation~\eqref{eq:kEq} is to couple the variables $x$ and $t$ so that the left-hand side of the equation becomes zero and the resulting equation involves the unknown~$F(0;t)$ only. 
%Knowing $F(0;t)$ we can determine $F(x;t)$ from equation~\eqref{eq:kEq}. 
\bigskip

Viewed as a polynomial in $x$, the kernel polynomial
\begin{equation*}
K(x,t) := x(1-t(x + \bar{x}))
\end{equation*} 
has two roots,
\begin{equation*}
x_0(t) = \frac{1 - \sqrt{1-4t^2}}{2t} \quad \text{and} \quad x_1(t) = \frac{1 + \sqrt{1-4t^2}}{2t},
\end{equation*}
only one of which, namely $x_0(t)$, can be interpreted as an element of $\mathbb{Q}[[t]]$. The composition $F(x_0(t),t)$ is therefore well-defined, and the substitution of $x_0(t)$ for $x$ in~\eqref{eq:kEq} results in
\begin{equation*}
0 = x_0(t) - tF(0;t).
\end{equation*}
Hence
\begin{equation*}
F(0;t) = \frac{x_0(t)}{t} \quad \text{and} \quad F(x;t) = \frac{1-\bar{x} x_0(t)}{1-t(x+\bar{x})}.
\end{equation*}
\bigskip

This can also be interpreted as follows: the right-hand side of equation~\eqref{eq:kEq} is a polynomial in $x$ over $\mathbb{Q}((t))$. Its degree is $1$, and so is its leading coefficient, and its only root is $t F(0;t)$. Since $x_0(t)$ 
%is a root of $K(x,t)$, and since it is a power series in $t$, and hence can be substituted for $x$ into $F(x;t)$, it
is a root of the left-hand side of the equation, it is necessarily also a root of its right-hand side, and therefore equal to $t F(0;t)$.
% of the equation. Consequently,
%\begin{equation*}
%t F(0;t) = x_0(t),
%\end{equation*}
%from which we deduce the expressions for $F(0;t)$ and $F(x;t)$ from before. 

\section{Wiener-Hopf factorization}\label{sec:wienerHopf}

Instead of solving equation~\eqref{eq:kEq} by eliminating the unknown on the left-hand side of the equation
%~$F(x;t)$ by annihilating the left-hand side of equation~\eqref{eq:kEq} 
one can just as well eliminate~$F(0;t)$ on the right-hand side. There is more than one way of doing so. The one we present here relates to what is known as Wiener-Hopf factorization. It is based on the factorization 
%It turns out that there is more than one way of doing so. One is based on the observation that 
\begin{equation*}
K(x,t) = - t (x- x_0)(x-x_1),
\end{equation*}
and the observation that 
\begin{equation*}
(x - x_0)^{-1} \in\mathbb{Q}[\bar{x}]((t)) \quad \text{and} \quad x-x_1 \in \mathbb{Q}[x]((t)).
\end{equation*}
%where $x - x_1$ is a series in $t$ which involves only non-negative powers of $x$ and the multiplicative inverse of $x - x_0$ is a series in $t$ which involves only negative powers of $x$. 
This can be exploited by dividing the kernel equation by $x-x_0$, resulting in
\begin{equation*}
-t (x-x_1) F(x;t) = \frac{1-\bar{x} t F(0;t)}{1-\bar{x} x_0},
\end{equation*}
an equation whose left-hand side lies in $\mathbb{Q}[x][[t]]$, and whose right-hand side is an element of~$\mathbb{Q}[\bar{x}][[t]]$ whose constant term with respect to $x$ is $1$. Discarding all terms which involve negative powers of $x$ gives an equation which does not involve $F(0;t)$ and from which we deduce that
\begin{equation*}
F(x;t) = -\frac{1}{t(x-x_1)}.
\end{equation*}

\section{Orbit-sum method}\label{sec:orbitSum}

Alternatively one can exploit that $1-t(x + \bar{x})$ is invariant under the natural action of the group generated by the transformation $x\mapsto \bar{x}$. The orbit of the kernel equation under this action consists of two equations, equation~\eqref{eq:kEq} and 
\begin{equation*}
\bar{x} (1 - t(x+\bar{x})) F(\bar{x};t) = \bar{x} - t F(0;t).
\end{equation*}
Subtracting the latter from the first and dividing by $x(1-t(x+\bar{x}))$ results in
\begin{equation*}
F(x;t) - \bar{x}^2 F(\bar{x};t) = \frac{1-\bar{x}^2}{1-t (x+\bar{x})}.
\end{equation*}
As $F(x;t)$ only involves non-negative powers of $x$, while $\bar{x}^2 F(\bar{x};t)$ only involves negative powers of $x$, we can eliminate $\bar{x}^2 F(\bar{x};t)$ by discarding all negative powers of $x$ and conclude that 
\begin{equation*}
F(x;t) = [x^\geq] \frac{1-\bar{x}^2}{1-t (x + \bar{x})}.
\end{equation*}

%The expression of $F(x;t)$ as the non-negative part of a rational function implies that $F(x;t)$ is D-finite, that is, satisfies a linear differential equation with polynomial coefficients~\cite{lipshitz}. 
The expression of $F(x;t)$ as the non-negative part of a rational function allows us to recover the expression from before. A partial fraction decomposition with respect to $x$ results in
\begin{equation*}
\frac{1-\bar{x}^2}{1-t (x+\bar{x})} = \frac{\bar{x}}{t} - \frac{\bar{x}^2 x_0}{t(1-\bar{x}x_0)} + \frac{\bar{x}}{t(1 - x\bar{x}_1)}
\end{equation*}
to which $[x^\geq ]$ easily applies when $ \frac{\bar{x}^2x_0}{1-\bar{x}x_0}$ and $\frac{\bar{x}}{1 - x\bar{x}_1}$ are interpreted as series that are elements of $\mathbb{Q}[x,\bar{x}][[t]]$.

\section{Compositional inverses}\label{sec:compInv}

%Recall that we started with the minimal polynomial of $F(x;t)$, derived a linear differential equation for $F(0;t)$, transformed it into a linear recurrence relation for its coefficients and then determined a closed form for them. In some situations one has a closed form for a sequence and wonders whether its generating function has a closed form as well and how it looks like in that case. Let us again have a look at the functional equation~\eqref{eq:kEq} for $F(x;t)$. 

%\begin{Definition}
%A series $F(x;t) \in\mathbb{Q}[x][[t]]$ is algebraic over $\mathbb{Q}(x,t)$, if there is a non-zero polynomial $P\in\mathbb{Q}[x,t][Y]$ such that 
%\begin{equation*}
%P(F(x;t)) = 0.
%\end{equation*}
%If $P(Y)$ has leading coefficient $1$, and its degree is minimal among all non-zero polynomials in $\mathbb{Q}[x,t][Y]$ which have $F(x;t)$ as a root, then $P(Y)$ is called the \emph{minimal polynomial} of $F(x;t)$. 
%\end{Definition}

So far we have considered the kernel polynomial as a polynomial in $x$, but we can just as well consider it as a polynomial in $t$. Its only root is then
\begin{equation*}
G(x) = \frac{x}{1+x^2}.
\end{equation*}
It can be viewed as an element of $x\mathbb{Q}[[x]]$. So the composition $F(x;G(x))$ is well-defined, and the substitution of $G(x)$ for $t$ in~\eqref{eq:kEq} results in 
\begin{equation}\label{eq:lagrange}
\frac{x}{1+x^2} F\left(0;\frac{x}{1+x^2}\right) = x.
\end{equation}
It follows that $G(x)$ is the compositional right-inverse of $x F(0;x)$. We next show how this implies that $F(0;t)$, and consequently $F(x;t)$, is algebraic.

\begin{Definition}\label{def:alg}
A series $F(x;t) \in\mathbb{Q}[x][[t]]$ is \emph{algebraic} over $\mathbb{Q}(x,t)$, if there is a non-zero polynomial $P\in\mathbb{Q}(x,t)[Y]$ such that 
\begin{equation*}
P(x,t,F(x;t)) = 0.
\end{equation*}
If $P$ has leading coefficient $1$, and its degree is minimal among all non-zero polynomials in $\mathbb{Q}(x,t)[Y]$ which have $F(x;t)$ as a root, then $P$ is called the \emph{minimal polynomial} of $F(x;t)$. 
\end{Definition}

Let 
\begin{equation*}
P(x,Y) = x - (1+x^2)Y
\end{equation*}
be the numerator of the minimal polynomial of $G(x)$. Then $P(x,G(x))=0$, and so $P(tF(0;t),t) = 0$ by replacing $x$ by $t F(0;t)$. The latter holds because $G(x)$ has valuation $1$, which guarantees that it has a right-inverse, and because the composition of series is associative, which implies that every right-inverse is also left-inverse. Therefore,~$F(0;t)$ is algebraic and a multiple of its minimal polynomial is 
\begin{equation*}
P(tY,t) = -t (1-Y+t^2Y^2). 
\end{equation*}
The expression for $F(0;t)$ and $F(x;t)$ from before can be recovered from that, and the coefficients of $F(0;t)$ can be computed from~\eqref{eq:lagrange} using \emph{Lagrange inversion}. For the convenience of the reader we recall Lagrange inversion here~\cite{GesselLagrange}.
\begin{Proposition}
Let $H(x)$ be a power series in $x$ whose valuation is $1$, and let $G(x)$ be its compositional inverse, that is, the series satisfying $G(H(x))=x$. Then
\begin{equation*}
[x^n] G^k(x) = \frac{k}{n}[x^{-k}] H^{-n}(x), \quad n\neq 0.
\end{equation*}
\end{Proposition}
Applying Lagrange inversion to $x/(1+x^2)$ we find 
\begin{equation}\label{eq:lagrange1}
f(0;n) = [t^{n+1}] t F(0;t) = \frac{1}{n+1} [t^{-1}] (t + t^{-1})^{n+1} =
 \begin{cases} 
   0,\quad  \text{if}\quad  n \not\equiv 0 \mod 2,\\
   \frac{1}{n/2+1} \binom{n}{n/2} \quad \text{otherwise}.
 \end{cases}
\end{equation}
%Obviously, $g(n) = f(0;2n)$ is hypergeometric and we can easily derive a linear recurrence relation of order $1$ from its expression,
%\begin{equation*}
%(4+2n) g(n+1) = 4(1+2n) g(n),
%\end{equation*}
%the same we found before. Multiplying it by $t^{n+1}$ and summing over all non-negative integers $n$ results in a linear differential equation for its generating function. Its inhomogeneity can be get rid of by differentiating, and doing so gives 
%\begin{equation*}
%(4t^2-t) \frac{\partial^2}{\partial t^2} + (10 t - 2) \frac{\partial}{\partial t} + 2.
%\end{equation*}
%To determine its solution space we can proceed as before and write it as the least common left multiple of two operators of order $1$ and determine for each of them its solution space. As before we can then find a closed form for the generating function $G(t)$ of $(g(n))$ by comparison of initial values,
%\begin{equation*}
%G(t) := \sum_{n=0}^\infty g(n)t^n = \frac{1-\sqrt{1-4t}}{2t},
%\end{equation*}
%from which we can immediately deduce a closed form for $F(0;t)$ and with equation~\eqref{eq:kEq} for $F(x;t)$. 

\section{Invariant method}\label{sec:invMethod}

The symmetry of the kernel polynomial the orbit-sum method is based on is not the only symmetry that can be exploited to solve the kernel equation. There is also a symmetry that appears in $K(x,t)$ having a (non-trivial) rational multiple in $\mathbb{Q}(x) + \mathbb{Q}(t)$. One such multiple is
\begin{equation}\label{eq:sepM}
-\frac{K(x,t)}{x t} = x + \bar{x} - \frac{1}{t},
\end{equation}
which shows that the subfields of the function field of the curve defined by $K(x,t)$ generated by $x$ and $t$, respectively, have a non-trivial intersection. It is generated by each of $t$ and $x + \bar{x}$. 
%In symbols: let $C$ be the curve defined by $K(x,t)$, and let $C(x,t)$ be its function field, that is, the quotient field of $\mathbb{Q}[x,t] / \langle K(x,t) \rangle$. Denoting by $C(x)$, $C(t)$ and $C(x+\bar{x})$ the subfields of $C(x,t)$ generated by $x$, $t$ and $x+\bar{x}$, respectively, we have  
%\begin{equation*}
%C(x) \cap C(t) = C(t) = C(x+\bar{x}).
%\end{equation*}
\bigskip

Now how is this useful? We can eliminate all terms involving $x$ from the right-hand side of equation~\eqref{eq:sepM} by subtracting the kernel equation and adding
\begin{equation*}
K(x,t) \frac{F(x;t)}{x t F(0;t)} = \frac{1}{tF(0;t)} - \bar{x}.
\end{equation*}
Multiplication of the resulting equation by $tF(0;t)$ gives
\begin{equation*}
K(x,t) \left( 
-\bar{x} F(0;t) - tF(0;t)F(x;t) + \bar{x} F(x;t) \right) = 1 - F(0;t) + t^2 F(0;t)^2
\end{equation*}
The valuations of $K(x;t)$ and $1 - F(0;t) + t^2 F(0;t)^2$ are zero with respect to both $x$ and $t$, and hence so are
the valuations of the coefficient of $K(x,t)$ on the left-hand side of the equation. Its right-hand side is therefore a multiple of $K(x,t)$ in $\mathbb{Q}[[x,t]]$ that does not depend on $x$. Consider now a monomial order on monomials in $x$ and $t$ with $x < t$. The smallest term in $K(x,t)$ is $x$, so the smallest term on the left-hand side of the equation is a multiple of $x$. But the right-hand side does not involve $x$. Thus it is zero and 
\begin{equation*}
1 - F(0;t) + t^2 F(0;t)^2 = 0.
\end{equation*}

\section{Continued fractions}\label{sec:contFrac}

%The special feature of (all variants of) the kernel method is that it allows one to solve the (seemingly) underdetermined kernel equation
%%\begin{equation*}
%%x(1-t(x+\bar{x}))F(x;t) = x - tF(0;t)
%%\end{equation*}
%by eliminating one of the unknowns. Instead, one can also set up an additional equation for~$F(x;t)$ and $F(0;t)$ and then solve the resulting system. It is not difficult to see that $F(0;t)$ satisfies the equation
The algebraic equation a generating function satisfies often has a combinatorial interpretation. The equation 
\begin{equation}\label{eq:returns}
F(0;t) = 1 + t^2 F(0;t)^2,
\end{equation}
reflects that any walk in $\mathbb{N}$ that starts and ends in~$0$ is either of length~$0$, or a walk that consists of a right step, followed by an excursion which does not visit~$0$, possibly of length $0$, followed by a left step, possibly continued by another excursion. 
%Equation~\eqref{eq:returns} has two solutions, only one of which can be interpreted as a power series in $t$, the one we found before for $F(0;t)$.
%, and knowing $F(0;t)$ we can determine $F(x;t)$ using the other equation. 
\begin{figure}[ht]
\hspace{-1.5cm}
  \centering
  \begin{center}
  \begin{tikzpicture}[scale=.4]
   
    \coordinate (Origin)   at (-2,-1);
    \coordinate (XAxisMin) at (-3,-1);
    \coordinate (XAxisMax) at (21.5,-1);
    \coordinate (YAxisMin) at (-2,-2);
    \coordinate (YAxisMax) at (-2,3.5);
    \draw [thin, gray,-latex] (XAxisMin) -- (XAxisMax);% Draw x axis
    \draw [thin, gray,-latex] (YAxisMin) -- (YAxisMax);% Draw y axis
 
    \coordinate (s1) at (1,1);
    \coordinate (s2) at (1,-1);

    \draw[style=help lines,dashed] (-2,-1) grid[step=1cm] (20,2);

    %draws the steps of the walk

%    \draw [thick,-latex,teal] (Origin)
%        -- ($(Origin)+(s1)$) ;
%
%    \draw [thick,-latex,teal] ($(Origin)+(s1)$)
%        -- ($(Origin)+(s1)+(s2)$) node [above right] {};
%
%    \draw [thick,-latex,teal] ($(Origin)+(s1)+(s2)$)
%        -- ($(Origin)+(s1)+(s2)+(s1)$) node [above right] {};
%
%    \draw [thick,-latex,teal] ($(Origin)+(s1)+(s2)+(s1)$)
%        -- ($(Origin)+(s1)+(s2)+(s1)+(s1)$) node [above right] {};
%
%    \draw [thick,-latex,teal] ($(Origin)+(s1)+(s2)+(s1)+(s1)$)
%        -- ($(Origin)+(s1)+(s2)+(s1)+(s1)+(s2)$) node [above right] {};
%
%    \draw [thick,-latex,teal] ($(Origin)+(s1)+(s2)+(s1)+(s1)+(s2)$)
%        -- ($(Origin)+(s1)+(s2)+(s1)+(s1)+(s2)+(s2)$) node [above right] {};
%        
%        \draw [thick,-latex,teal] ($(Origin)+(s1)+(s2)+(s1)+(s1)+(s2)+(s2)$)
%        -- ($(Origin)+(s1)+(s2)+(s1)+(s1)+(s2)+(s2)+(s2)$) node [above right] {};
        
           \draw [thick,-latex, blue] ($(Origin)$)
        -- ($(Origin)+(s1)$) node [above right] {};
        
               \draw [thick,-latex,teal] ($(Origin)+(s1)$)
        -- ($(Origin)+(s1)+(s1)$) node [above right] {};

               \draw [thick,-latex,teal] ($(Origin)+(s1)+(s1)$)
        -- ($(Origin)+(s1)+(s1)+(s2)$) node [above right] {};
        
                \draw [thick,-latex,teal] ($(Origin)+(s1)+(s1)+(s2)$)
        -- ($(Origin)+(s1)+(s1)+(s2)+(s1)$) node [above right] {};
        
                 \draw [thick,-latex,teal] ($(Origin)+(s1)+(s1)+(s2)+(s1)$)
        -- ($(Origin)+(s1)+(s1)+(s2)+(s1)+(s2)$) node [above right] {};
        
        \draw [thick,-latex,teal] ($(Origin)+(s1)+(s1)+(s2)+(s1)+(s2)$)
        -- ($(Origin)+(s1)+(s1)+(s2)+(s1)+(s2)+(s1)$) node [above right] {};
        
           \draw [thick,-latex,teal] ($(Origin)+(s1)+(s1)+(s2)+(s1)+(s2)+(s1)$)
        -- ($(Origin)+(s1)+(s1)+(s2)+(s1)+(s2)+(s1)+(s1)$) node [above right] {};
        
            \draw [thick,-latex,teal] ($(Origin)+(s1)+(s1)+(s2)+(s1)+(s2)+(s1)+(s1)$)
        -- ($(Origin)+(s1)+(s1)+(s2)+(s1)+(s2)+(s1)+(s1)+(s2)$) node [above right] {};
   
     \draw [thick,-latex,teal] ($(Origin)+(s1)+(s1)+(s2)+(s1)+(s2)+(s1)+(s1)+(s2)$)
        -- ($(Origin)+(s1)+(s1)+(s2)+(s1)+(s2)+(s1)+(s1)+(s2)+(s2)$) node [above right] {};
        
         \draw [thick,-latex, blue] ($(Origin)+(s1)+(s1)+(s2)+(s1)+(s2)+(s1)+(s1)+(s2)+(s2)$)
        -- ($(Origin)+(s1)+(s1)+(s2)+(s1)+(s2)+(s1)+(s1)+(s2)+(s2)+(s2)$) node [above right] {};
        
             \draw [thick,-latex,violet] ($(Origin)+(s1)+(s1)+(s2)+(s1)+(s2)+(s1)+(s1)+(s2)+(s2)+(s2)$)
        -- ($(Origin)+(s1)+(s1)+(s2)+(s1)+(s2)+(s1)+(s1)+(s2)+(s2)+(s2)+(s1)$) node [above right] {};
        
         \draw [thick,-latex,violet] ($(Origin)+(s1)+(s1)+(s2)+(s1)+(s2)+(s1)+(s1)+(s2)+(s2)+(s2)+(s1)$)
        -- ($(Origin)+(s1)+(s1)+(s2)+(s1)+(s2)+(s1)+(s1)+(s2)+(s2)+(s2)+(s1)+(s1)$) node [above right] {};
        
         \draw [thick,-latex,violet] ($(Origin)+(s1)+(s1)+(s2)+(s1)+(s2)+(s1)+(s1)+(s2)+(s2)+(s2)+(s1)+(s1)$)
        -- ($(Origin)+(s1)+(s1)+(s2)+(s1)+(s2)+(s1)+(s1)+(s2)+(s2)+(s2)+(s1)+(s1)+(s2)$) node [above right] {};

 \draw [thick,-latex,violet] ($(Origin)+(s1)+(s1)+(s2)+(s1)+(s2)+(s1)+(s1)+(s2)+(s2)+(s2)+(s1)+(s1)+(s2)$)
        -- ($(Origin)+(s1)+(s1)+(s2)+(s1)+(s2)+(s1)+(s1)+(s2)+(s2)+(s2)+(s1)+(s1)+(s2)+(s1)$) node [above right] {};
        
         \draw [thick,-latex,violet] ($(Origin)+(s1)+(s1)+(s2)+(s1)+(s2)+(s1)+(s1)+(s2)+(s2)+(s2)+(s1)+(s1)+(s2)+(s1)$)
        -- ($(Origin)+(s1)+(s1)+(s2)+(s1)+(s2)+(s1)+(s1)+(s2)+(s2)+(s2)+(s1)+(s1)+(s2)+(s1)+(s2)$) node [above right] {};

    \draw [thick,-latex,violet] ($(Origin)+(s1)+(s1)+(s2)+(s1)+(s2)+(s1)+(s1)+(s2)+(s2)+(s2)+(s1)+(s1)+(s2)+(s1)+(s2)$)
        -- ($(Origin)+(s1)+(s1)+(s2)+(s1)+(s2)+(s1)+(s1)+(s2)+(s2)+(s2)+(s1)+(s1)+(s2)+(s1)+(s2)+(s1)$) node [above right] {};
        
          \draw [thick,-latex,violet] ($(Origin)+(s1)+(s1)+(s2)+(s1)+(s2)+(s1)+(s1)+(s2)+(s2)+(s2)+(s1)+(s1)+(s2)+(s1)+(s2)+(s1)$)
        -- ($(Origin)+(s1)+(s1)+(s2)+(s1)+(s2)+(s1)+(s1)+(s2)+(s2)+(s2)+(s1)+(s1)+(s2)+(s1)+(s2)+(s1)+(s1)$) node [above right] {};
        
            \draw [thick,-latex,violet] ($(Origin)+(s1)+(s1)+(s2)+(s1)+(s2)+(s1)+(s1)+(s2)+(s2)+(s2)+(s1)+(s1)+(s2)+(s1)+(s2)+(s1)+(s1)$)
        -- ($(Origin)+(s1)+(s1)+(s2)+(s1)+(s2)+(s1)+(s1)+(s2)+(s2)+(s2)+(s1)+(s1)+(s2)+(s1)+(s2)+(s1)+(s1)+(s2)$) node [above right] {};

    \draw [thick,-latex,violet] ($(Origin)+(s1)+(s1)+(s2)+(s1)+(s2)+(s1)+(s1)+(s2)+(s2)+(s2)+(s1)+(s1)+(s2)+(s1)+(s2)+(s1)+(s1)+(s2)$)
        -- ($(Origin)+(s1)+(s1)+(s2)+(s1)+(s2)+(s1)+(s1)+(s2)+(s2)+(s2)+(s1)+(s1)+(s2)+(s1)+(s2)+(s1)+(s1)+(s2)+(s2)$) node [above right] {};
        
         \draw [thick,-latex,violet] ($(Origin)+(s1)+(s1)+(s2)+(s1)+(s2)+(s1)+(s1)+(s2)+(s2)+(s2)+(s1)+(s1)+(s2)+(s1)+(s2)+(s1)+(s1)+(s2) +(s2)$)
        -- ($(Origin)+(s1)+(s1)+(s2)+(s1)+(s2)+(s1)+(s1)+(s2)+(s2)+(s2)+(s1)+(s1)+(s2)+(s1)+(s2)+(s1)+(s1)+(s2)+(s2) + (s2)$) node [above right] {};
        
         \draw [thick,-latex,violet] ($(Origin)+(s1)+(s1)+(s2)+(s1)+(s2)+(s1)+(s1)+(s2)+(s2)+(s2)+(s1)+(s1)+(s2)+(s1)+(s2)+(s1)+(s1)+(s2) +(s2)+(s2)$)
        -- ($(Origin)+(s1)+(s1)+(s2)+(s1)+(s2)+(s1)+(s1)+(s2)+(s2)+(s2)+(s1)+(s1)+(s2)+(s1)+(s2)+(s1)+(s1)+(s2)+(s2) + (s2)+(s1)$) node [above right] {};
        
               \draw [thick,-latex,violet] ($(Origin)+(s1)+(s1)+(s2)+(s1)+(s2)+(s1)+(s1)+(s2)+(s2)+(s2)+(s1)+(s1)+(s2)+(s1)+(s2)+(s1)+(s1)+(s2) +(s2)+(s2)+(s1)$)
        -- ($(Origin)+(s1)+(s1)+(s2)+(s1)+(s2)+(s1)+(s1)+(s2)+(s2)+(s2)+(s1)+(s1)+(s2)+(s1)+(s2)+(s1)+(s1)+(s2)+(s2) + (s2)+(s1)+(s2)$) node [above right] {};

  \end{tikzpicture}
  \end{center}
  \caption{A lattice walk and its ``first passage'' decomposition.}
 \end{figure}
This decomposition of a lattice walk is referred to as the ``first passage'' decomposition. It is distinguished from the ``arch'' decomposition. The latter decomposes a path into a sequence of triples consisting of a right step, a path which does not visit $0$ and a left step. 
\begin{figure}[ht]
\hspace{-1.5cm}
  \centering
  \begin{center}
  \begin{tikzpicture}[scale=.4]
   
    \coordinate (Origin)   at (-2,-1);
    \coordinate (XAxisMin) at (-3,-1);
    \coordinate (XAxisMax) at (21.5,-1);
    \coordinate (YAxisMin) at (-2,-2);
    \coordinate (YAxisMax) at (-2,3.5);
    \draw [thin, gray,-latex] (XAxisMin) -- (XAxisMax);% Draw x axis
    \draw [thin, gray,-latex] (YAxisMin) -- (YAxisMax);% Draw y axis
 
    \coordinate (s1) at (1,1);
    \coordinate (s2) at (1,-1);

    \draw[style=help lines,dashed] (-2,-1) grid[step=1cm] (20,2);

    %draws the steps of the walk

%    \draw [thick,-latex,teal] (Origin)
%        -- ($(Origin)+(s1)$) ;
%
%    \draw [thick,-latex,teal] ($(Origin)+(s1)$)
%        -- ($(Origin)+(s1)+(s2)$) node [above right] {};
%
%    \draw [thick,-latex,teal] ($(Origin)+(s1)+(s2)$)
%        -- ($(Origin)+(s1)+(s2)+(s1)$) node [above right] {};
%
%    \draw [thick,-latex,teal] ($(Origin)+(s1)+(s2)+(s1)$)
%        -- ($(Origin)+(s1)+(s2)+(s1)+(s1)$) node [above right] {};
%
%    \draw [thick,-latex,teal] ($(Origin)+(s1)+(s2)+(s1)+(s1)$)
%        -- ($(Origin)+(s1)+(s2)+(s1)+(s1)+(s2)$) node [above right] {};
%
%    \draw [thick,-latex,teal] ($(Origin)+(s1)+(s2)+(s1)+(s1)+(s2)$)
%        -- ($(Origin)+(s1)+(s2)+(s1)+(s1)+(s2)+(s2)$) node [above right] {};
%        
%        \draw [thick,-latex,teal] ($(Origin)+(s1)+(s2)+(s1)+(s1)+(s2)+(s2)$)
%        -- ($(Origin)+(s1)+(s2)+(s1)+(s1)+(s2)+(s2)+(s2)$) node [above right] {};
        
           \draw [thick,-latex,blue] ($(Origin)$)
        -- ($(Origin)+(s1)$) node [above right] {};
        
               \draw [thick,-latex,teal] ($(Origin)+(s1)$)
        -- ($(Origin)+(s1)+(s1)$) node [above right] {};

               \draw [thick,-latex,teal] ($(Origin)+(s1)+(s1)$)
        -- ($(Origin)+(s1)+(s1)+(s2)$) node [above right] {};
        
                \draw [thick,-latex,teal] ($(Origin)+(s1)+(s1)+(s2)$)
        -- ($(Origin)+(s1)+(s1)+(s2)+(s1)$) node [above right] {};
        
                 \draw [thick,-latex,teal] ($(Origin)+(s1)+(s1)+(s2)+(s1)$)
        -- ($(Origin)+(s1)+(s1)+(s2)+(s1)+(s2)$) node [above right] {};
        
        \draw [thick,-latex,teal] ($(Origin)+(s1)+(s1)+(s2)+(s1)+(s2)$)
        -- ($(Origin)+(s1)+(s1)+(s2)+(s1)+(s2)+(s1)$) node [above right] {};
        
           \draw [thick,-latex,teal] ($(Origin)+(s1)+(s1)+(s2)+(s1)+(s2)+(s1)$)
        -- ($(Origin)+(s1)+(s1)+(s2)+(s1)+(s2)+(s1)+(s1)$) node [above right] {};
        
            \draw [thick,-latex,teal] ($(Origin)+(s1)+(s1)+(s2)+(s1)+(s2)+(s1)+(s1)$)
        -- ($(Origin)+(s1)+(s1)+(s2)+(s1)+(s2)+(s1)+(s1)+(s2)$) node [above right] {};
   
     \draw [thick,-latex,teal] ($(Origin)+(s1)+(s1)+(s2)+(s1)+(s2)+(s1)+(s1)+(s2)$)
        -- ($(Origin)+(s1)+(s1)+(s2)+(s1)+(s2)+(s1)+(s1)+(s2)+(s2)$) node [above right] {};
        
         \draw [thick,-latex,blue] ($(Origin)+(s1)+(s1)+(s2)+(s1)+(s2)+(s1)+(s1)+(s2)+(s2)$)
        -- ($(Origin)+(s1)+(s1)+(s2)+(s1)+(s2)+(s1)+(s1)+(s2)+(s2)+(s2)$) node [above right] {};
        
             \draw [thick,-latex,blue] ($(Origin)+(s1)+(s1)+(s2)+(s1)+(s2)+(s1)+(s1)+(s2)+(s2)+(s2)$)
        -- ($(Origin)+(s1)+(s1)+(s2)+(s1)+(s2)+(s1)+(s1)+(s2)+(s2)+(s2)+(s1)$) node [above right] {};
        
         \draw [thick,-latex,teal] ($(Origin)+(s1)+(s1)+(s2)+(s1)+(s2)+(s1)+(s1)+(s2)+(s2)+(s2)+(s1)$)
        -- ($(Origin)+(s1)+(s1)+(s2)+(s1)+(s2)+(s1)+(s1)+(s2)+(s2)+(s2)+(s1)+(s1)$) node [above right] {};
        
         \draw [thick,-latex,teal] ($(Origin)+(s1)+(s1)+(s2)+(s1)+(s2)+(s1)+(s1)+(s2)+(s2)+(s2)+(s1)+(s1)$)
        -- ($(Origin)+(s1)+(s1)+(s2)+(s1)+(s2)+(s1)+(s1)+(s2)+(s2)+(s2)+(s1)+(s1)+(s2)$) node [above right] {};

 \draw [thick,-latex,teal] ($(Origin)+(s1)+(s1)+(s2)+(s1)+(s2)+(s1)+(s1)+(s2)+(s2)+(s2)+(s1)+(s1)+(s2)$)
        -- ($(Origin)+(s1)+(s1)+(s2)+(s1)+(s2)+(s1)+(s1)+(s2)+(s2)+(s2)+(s1)+(s1)+(s2)+(s1)$) node [above right] {};
        
         \draw [thick,-latex,teal] ($(Origin)+(s1)+(s1)+(s2)+(s1)+(s2)+(s1)+(s1)+(s2)+(s2)+(s2)+(s1)+(s1)+(s2)+(s1)$)
        -- ($(Origin)+(s1)+(s1)+(s2)+(s1)+(s2)+(s1)+(s1)+(s2)+(s2)+(s2)+(s1)+(s1)+(s2)+(s1)+(s2)$) node [above right] {};

    \draw [thick,-latex,teal] ($(Origin)+(s1)+(s1)+(s2)+(s1)+(s2)+(s1)+(s1)+(s2)+(s2)+(s2)+(s1)+(s1)+(s2)+(s1)+(s2)$)
        -- ($(Origin)+(s1)+(s1)+(s2)+(s1)+(s2)+(s1)+(s1)+(s2)+(s2)+(s2)+(s1)+(s1)+(s2)+(s1)+(s2)+(s1)$) node [above right] {};
        
          \draw [thick,-latex,teal] ($(Origin)+(s1)+(s1)+(s2)+(s1)+(s2)+(s1)+(s1)+(s2)+(s2)+(s2)+(s1)+(s1)+(s2)+(s1)+(s2)+(s1)$)
        -- ($(Origin)+(s1)+(s1)+(s2)+(s1)+(s2)+(s1)+(s1)+(s2)+(s2)+(s2)+(s1)+(s1)+(s2)+(s1)+(s2)+(s1)+(s1)$) node [above right] {};
        
            \draw [thick,-latex,teal] ($(Origin)+(s1)+(s1)+(s2)+(s1)+(s2)+(s1)+(s1)+(s2)+(s2)+(s2)+(s1)+(s1)+(s2)+(s1)+(s2)+(s1)+(s1)$)
        -- ($(Origin)+(s1)+(s1)+(s2)+(s1)+(s2)+(s1)+(s1)+(s2)+(s2)+(s2)+(s1)+(s1)+(s2)+(s1)+(s2)+(s1)+(s1)+(s2)$) node [above right] {};

    \draw [thick,-latex,teal] ($(Origin)+(s1)+(s1)+(s2)+(s1)+(s2)+(s1)+(s1)+(s2)+(s2)+(s2)+(s1)+(s1)+(s2)+(s1)+(s2)+(s1)+(s1)+(s2)$)
        -- ($(Origin)+(s1)+(s1)+(s2)+(s1)+(s2)+(s1)+(s1)+(s2)+(s2)+(s2)+(s1)+(s1)+(s2)+(s1)+(s2)+(s1)+(s1)+(s2)+(s2)$) node [above right] {};
        
         \draw [thick,-latex,blue] ($(Origin)+(s1)+(s1)+(s2)+(s1)+(s2)+(s1)+(s1)+(s2)+(s2)+(s2)+(s1)+(s1)+(s2)+(s1)+(s2)+(s1)+(s1)+(s2) +(s2)$)
        -- ($(Origin)+(s1)+(s1)+(s2)+(s1)+(s2)+(s1)+(s1)+(s2)+(s2)+(s2)+(s1)+(s1)+(s2)+(s1)+(s2)+(s1)+(s1)+(s2)+(s2) + (s2)$) node [above right] {};
        
         \draw [thick,-latex,blue] ($(Origin)+(s1)+(s1)+(s2)+(s1)+(s2)+(s1)+(s1)+(s2)+(s2)+(s2)+(s1)+(s1)+(s2)+(s1)+(s2)+(s1)+(s1)+(s2) +(s2)+(s2)$)
        -- ($(Origin)+(s1)+(s1)+(s2)+(s1)+(s2)+(s1)+(s1)+(s2)+(s2)+(s2)+(s1)+(s1)+(s2)+(s1)+(s2)+(s1)+(s1)+(s2)+(s2) + (s2)+(s1)$) node [above right] {};
        
               \draw [thick,-latex,blue] ($(Origin)+(s1)+(s1)+(s2)+(s1)+(s2)+(s1)+(s1)+(s2)+(s2)+(s2)+(s1)+(s1)+(s2)+(s1)+(s2)+(s1)+(s1)+(s2) +(s2)+(s2)+(s1)$)
        -- ($(Origin)+(s1)+(s1)+(s2)+(s1)+(s2)+(s1)+(s1)+(s2)+(s2)+(s2)+(s1)+(s1)+(s2)+(s1)+(s2)+(s1)+(s1)+(s2)+(s2) + (s2)+(s1)+(s2)$) node [above right] {};

  \end{tikzpicture}
  \end{center}
  \caption{A lattice walk and its ``arch'' decomposition.}
 \end{figure}
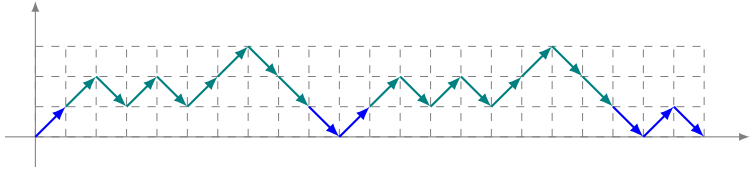
It gives rise to the following equivalent equation
\begin{equation*}
F(0;t) = \frac{1}{1- t^2 F(0;t)}.
\end{equation*}
Applied recursively it implies a representation of $F(0;t)$ as the \emph{continued fraction}
\begin{equation}\label{eq:contFraction}
F(0;t) = \frac{1}{1-\frac{t^2}{1-\frac{t^2}{1-\dots}}}.
\end{equation}
%This expression is cumbrous, it is therefore more convenient to write 
%\begin{equation*}
%F(0;t) = \frac{1}{1-}\frac{t^2}{1-}\frac{t^2}{1-\dots}
%\end{equation*}
%instead. 
The finite continued fractions $(F_k)$ defined by 
\begin{equation*}
F_0 = 1, \quad \text{and} \quad F_k = \frac{1}{1-t^2 F_{k-1}}, \quad k\geq 1,
\end{equation*}
are called the \emph{convergents} of the continued fraction~\eqref{eq:contFraction}. Of course the convergents $(F_k)$ converge to $F(0;t)$ in the sense that  
%\begin{equation*}
%\lim_{k \to \infty} F_k = F(0;t)
%\end{equation*}
\begin{equation*}
\lim_{k \to \infty} \val (F(0;t) - F_k ) = 0.
\end{equation*}
The convergents have a combinatorial meaning too: $F_k$ is the generating function of excursions on $\{0, 1, \dots, k\}$ that start at $0$ and take their steps from $\{-1,1\}$. This is clearly true for $F_0$, and since any excursion on $\{0,1, \dots,k\}$ is a sequence of walks that start and end with a right and left step, respectively, and whose other steps form a walk on $\{0,\dots,k-1\}$, it is easily seen that it also holds for $k \geq 1$. The question how to compute the coefficients of the series expansion of $F_k$ efficiently relates the enumeration of lattice walks to the theory of continued fractions and orthogonal polynomials.

\section{Reflection principle}\label{sec:reflecPrinciple}
It is natural to ask whether also the other expressions we have found for $F(x;t)$ and its evaluation give more insight into the combinatorial problem, apart from the fact that they represent the generating function of the counting sequence. It is obvious how to interpret the right-hand side of
\begin{equation*}
F(x;t) =  \frac{1-\bar{x} x_0}{1-t(x+\bar{x})} 
\end{equation*}
combinatorially having in mind that $F(0;t) = x_0(t)/t$ is the generating function of walks in $\mathbb{N}$ that start and end in $0$ and $\frac{1}{1-t(x+\bar{x})}$ represents the generating function of walks in $\mathbb{Z}$ that start at $0$. It is the generating function of walks in $\mathbb{Z}$ that start at $0$ minus the generating function of the subset of those walks which do not lie in $\mathbb{N}$. It might be less obvious how to interpret the right-hand side of 
\begin{equation}\label{eq:posPart}
F(x;t) = [x^\geq] \frac{1-\bar{x}^2}{1-t (x+\bar{x})}
\end{equation}
in this respect. The right-hand side of the equation reflects the existence of a bijection between walks in $\mathbb{Z}$ that start at $0$, end in $\mathbb{N}$ but do not lie in $\mathbb{N}$ and walks that start at $-2$ and end in $\mathbb{N}$. It is defined by reflecting the initial part of a walk up to its last visit of $-1$ at $-1$. See the figure below for an illustration. As the bijection neither affects the length nor the end point of a walk, this provides another argument why identity~\eqref{eq:posPart} holds. 

\begin{figure}[ht]
\hspace{-1.5cm}
  \centering
  \begin{center}
  \begin{tikzpicture}[scale=.4]
   
    \coordinate (Origin)   at (-2,-1);
    \coordinate (XAxisMin) at (-2,-1);
    \coordinate (XAxisMax) at (21.5,-1);
    \coordinate (YAxisMin) at (-2,-6);
    \coordinate (YAxisMax) at (-2,2.5);
    \draw [thin, gray,-latex] (XAxisMin) -- (XAxisMax);% Draw x axis
    \draw [thin, gray,-latex] (YAxisMin) -- (YAxisMax);% Draw y axis

\draw [thin, color=red!60] (-2,-2) -- (20, -2);

    \coordinate (s1) at (1,1);
    \coordinate (s2) at (1,-1);

    \draw[style=help lines,dashed] (-2,-5) grid[step=1cm] (20,1);

    %draws the steps of the walk

    \draw [thick,-latex,teal] (Origin)
        -- ($(Origin)+(s1)$) ;

    \draw [thick,-latex,teal] ($(Origin)+(s1)$)
        -- ($(Origin)+(s1)+(s2)$) node [above right] {};

    \draw [thick,-latex,teal] ($(Origin)+(s1)+(s2)$)
        -- ($(Origin)+(s1)+(s2)+(s1)$) node [above right] {};

    \draw [thick,-latex,teal] ($(Origin)+(s1)+(s2)+(s1)$)
        -- ($(Origin)+(s1)+(s2)+(s1)+(s1)$) node [above right] {};

    \draw [thick,-latex,teal] ($(Origin)+(s1)+(s2)+(s1)+(s1)$)
        -- ($(Origin)+(s1)+(s2)+(s1)+(s1)+(s2)$) node [above right] {};

    \draw [thick,-latex,teal] ($(Origin)+(s1)+(s2)+(s1)+(s1)+(s2)$)
        -- ($(Origin)+(s1)+(s2)+(s1)+(s1)+(s2)+(s2)$) node [above right] {};
        
        \draw [thick,-latex,teal] ($(Origin)+(s1)+(s2)+(s1)+(s1)+(s2)+(s2)$)
        -- ($(Origin)+(s1)+(s2)+(s1)+(s1)+(s2)+(s2)+(s2)$) node [above right] {};

              \draw [thick,-latex,teal] ($(Origin)+(s1)+(s2)+(s1)+(s1)+(s2)+(s2)+(s2)$)
        -- ($(Origin)+(s1)+(s2)+(s1)+(s1)+(s2)+(s2)+(s2)+(s1)$) node [above right] {};

            \draw [thick,-latex,teal] ($(Origin)+(s1)+(s2)+(s1)+(s1)+(s2)+(s2)+(s2)+(s1)$)
        -- ($(Origin)+(s1)+(s2)+(s1)+(s1)+(s2)+(s2)+(s2)+(s1)+(s2)$) node [above right] {};

           \draw [thick,-latex,teal] ($(Origin)+(s1)+(s2)+(s1)+(s1)+(s2)+(s2)+(s2)+(s1)+(s2)$)
        -- ($(Origin)+(s1)+(s2)+(s1)+(s1)+(s2)+(s2)+(s2)+(s1)+(s2)+(s1)$) node [above right] {};
        
               \draw [thick,-latex,teal] ($(Origin)+(s1)+(s2)+(s1)+(s1)+(s2)+(s2)+(s2)+(s1)+(s2)+(s1)$)
        -- ($(Origin)+(s1)+(s2)+(s1)+(s1)+(s2)+(s2)+(s2)+(s1)+(s2)+(s1)+(s1)$) node [above right] {};

               \draw [thick,-latex,teal] ($(Origin)+(s1)+(s2)+(s1)+(s1)+(s2)+(s2)+(s2)+(s1)+(s2)+(s1)+(s1)$)
        -- ($(Origin)+(s1)+(s2)+(s1)+(s1)+(s2)+(s2)+(s2)+(s1)+(s2)+(s1)+(s1)+(s2)$) node [above right] {};
        
                \draw [thick,-latex,teal] ($(Origin)+(s1)+(s2)+(s1)+(s1)+(s2)+(s2)+(s2)+(s1)+(s2)+(s1)+(s1)+(s2)$)
        -- ($(Origin)+(s1)+(s2)+(s1)+(s1)+(s2)+(s2)+(s2)+(s1)+(s2)+(s1)+(s1)+(s2)+(s1)$) node [above right] {};
        
                 \draw [thick,-latex,teal] ($(Origin)+(s1)+(s2)+(s1)+(s1)+(s2)+(s2)+(s2)+(s1)+(s2)+(s1)+(s1)+(s2)+(s1)$)
        -- ($(Origin)+(s1)+(s2)+(s1)+(s1)+(s2)+(s2)+(s2)+(s1)+(s2)+(s1)+(s1)+(s2)+(s1)+(s2)$) node [above right] {};
        
        \draw [thick,-latex,teal] ($(Origin)+(s1)+(s2)+(s1)+(s1)+(s2)+(s2)+(s2)+(s1)+(s2)+(s1)+(s1)+(s2)+(s1)+(s2)$)
        -- ($(Origin)+(s1)+(s2)+(s1)+(s1)+(s2)+(s2)+(s2)+(s1)+(s2)+(s1)+(s1)+(s2)+(s1)+(s2)+(s1)$) node [above right] {};
        
           \draw [thick,-latex,teal] ($(Origin)+(s1)+(s2)+(s1)+(s1)+(s2)+(s2)+(s2)+(s1)+(s2)+(s1)+(s1)+(s2)+(s1)+(s2)+(s1)$)
        -- ($(Origin)+(s1)+(s2)+(s1)+(s1)+(s2)+(s2)+(s2)+(s1)+(s2)+(s1)+(s1)+(s2)+(s1)+(s2)+(s1)+(s1)$) node [above right] {};
        
            \draw [thick,-latex,teal] ($(Origin)+(s1)+(s2)+(s1)+(s1)+(s2)+(s2)+(s2)+(s1)+(s2)+(s1)+(s1)+(s2)+(s1)+(s2)+(s1)+(s1)$)
        -- ($(Origin)+(s1)+(s2)+(s1)+(s1)+(s2)+(s2)+(s2)+(s1)+(s2)+(s1)+(s1)+(s2)+(s1)+(s2)+(s1)+(s1)+(s2)$) node [above right] {};
   
     \draw [thick,-latex,teal] ($(Origin)+(s1)+(s2)+(s1)+(s1)+(s2)+(s2)+(s2)+(s1)+(s2)+(s1)+(s1)+(s2)+(s1)+(s2)+(s1)+(s1)+(s2)$)
        -- ($(Origin)+(s1)+(s2)+(s1)+(s1)+(s2)+(s2)+(s2)+(s1)+(s2)+(s1)+(s1)+(s2)+(s1)+(s2)+(s1)+(s1)+(s2)+(s2)$) node [above right] {};
        
         \draw [thick,-latex,teal] ($(Origin)+(s1)+(s2)+(s1)+(s1)+(s2)+(s2)+(s2)+(s1)+(s2)+(s1)+(s1)+(s2)+(s1)+(s2)+(s1)+(s1)+(s2)+(s2)$)
        -- ($(Origin)+(s1)+(s2)+(s1)+(s1)+(s2)+(s2)+(s2)+(s1)+(s2)+(s1)+(s1)+(s2)+(s1)+(s2)+(s1)+(s1)+(s2)+(s2)+(s2)$) node [above right] {};
        
             \draw [thick,-latex,teal] ($(Origin)+(s1)+(s2)+(s1)+(s1)+(s2)+(s2)+(s2)+(s1)+(s2)+(s1)+(s1)+(s2)+(s1)+(s2)+(s1)+(s1)+(s2)+(s2)+(s2)$)
        -- ($(Origin)+(s1)+(s2)+(s1)+(s1)+(s2)+(s2)+(s2)+(s1)+(s2)+(s1)+(s1)+(s2)+(s1)+(s2)+(s1)+(s1)+(s2)+(s2)+(s2)+(s1)$) node [above right] {};
        
         \draw [thick,-latex,teal] ($(Origin)+(s1)+(s2)+(s1)+(s1)+(s2)+(s2)+(s2)+(s1)+(s2)+(s1)+(s1)+(s2)+(s1)+(s2)+(s1)+(s1)+(s2)+(s2)+(s2)+(s1)$)
        -- ($(Origin)+(s1)+(s2)+(s1)+(s1)+(s2)+(s2)+(s2)+(s1)+(s2)+(s1)+(s1)+(s2)+(s1)+(s2)+(s1)+(s1)+(s2)+(s2)+(s2)+(s1)+(s1)$) node [above right] {};
        
         \draw [thick,-latex,teal] ($(Origin)+(s1)+(s2)+(s1)+(s1)+(s2)+(s2)+(s2)+(s1)+(s2)+(s1)+(s1)+(s2)+(s1)+(s2)+(s1)+(s1)+(s2)+(s2)+(s2)+(s1)+(s1)$)
        -- ($(Origin)+(s1)+(s2)+(s1)+(s1)+(s2)+(s2)+(s2)+(s1)+(s2)+(s1)+(s1)+(s2)+(s1)+(s2)+(s1)+(s1)+(s2)+(s2)+(s2)+(s1)+(s1)+(s2)$) node [above right] {};

  \draw [thick,-latex, color=red!60] (-2,-3)
        -- ($(-2,-3)+(s2)$) ;

    \draw [thick,-latex, color=red!60] ($(-2,-3)+(s2)$)
        -- ($(-2,-3)+(s2)+(s1)$) node [above right] {};

    \draw [thick,-latex, color=red!60] ($(-2,-3)+(s2)+(s1)$)
        -- ($(-2,-3)+(s2)+(s1)+(s2)$) node [above right] {};

    \draw [thick,-latex, color=red!60] ($(-2,-3)+(s2)+(s1)+(s2)$)
        -- ($(-2,-3)+(s2)+(s1)+(s2)+(s2)$) node [above right] {};

    \draw [thick,-latex, color=red!60] ($(-2,-3)+(s2)+(s1)+(s2)+(s2)$)
        -- ($(-2,-3)+(s2)+(s1)+(s2)+(s2)+(s1)$) node [above right] {};

    \draw [thick,-latex, color=red!60] ($(-2,-3)+(s2)+(s1)+(s2)+(s2)+(s1)$)
        -- ($(-2,-3)+(s2)+(s1)+(s2)+(s2)+(s1)+(s1)$) node [above right] {};
        
        \draw [thick,-latex, color=red!60] ($(-2,-3)+(s2)+(s1)+(s2)+(s2)+(s1)+(s1)$)
        -- ($(-2,-3)+(s2)+(s1)+(s2)+(s2)+(s1)+(s1)+(s1)$) node [above right] {};
        
             \draw [thick,-latex, color=red!60] ($(-2,-3)+(s2)+(s1)+(s2)+(s2)+(s1)+(s1)+(s1)$)
        -- ($(-2,-3)+(s2)+(s1)+(s2)+(s2)+(s1)+(s1)+(s1)+(s2)$) node [above right] {};
        
          \draw [thick,-latex, color=red!60] ($(-2,-3)+(s2)+(s1)+(s2)+(s2)+(s1)+(s1)+(s1)+(s2)$)
        -- ($(-2,-3)+(s2)+(s1)+(s2)+(s2)+(s1)+(s1)+(s1)+(s2)+(s1)$) node [above right] {};
        
           \draw [thick,-latex, color=red!60] ($(-2,-3)+(s2)+(s1)+(s2)+(s2)+(s1)+(s1)+(s1)+(s2)+(s1)$)
        -- ($(-2,-3)+(s2)+(s1)+(s2)+(s2)+(s1)+(s1)+(s1)+(s2)+(s1)+(s2)$) node [above right] {};
        
               \draw [thick,-latex, color=red!60] ($((-2,-3)+(s2)+(s1)+(s2)+(s2)+(s1)+(s1)+(s1)+(s2)+(s1)+(s2)$)
        -- ($(-2,-3)+(s2)+(s1)+(s2)+(s2)+(s1)+(s1)+(s1)+(s2)+(s1)+(s2)+(s2)$) node [above right] {};

               \draw [thick,-latex, color=red!60] ($(-2,-3)+(s2)+(s1)+(s2)+(s2)+(s1)+(s1)+(s1)+(s2)+(s1)+(s2)+(s2)$)
        -- ($(-2,-3)+(s2)+(s1)+(s2)+(s2)+(s1)+(s1)+(s1)+(s2)+(s1)+(s2)+(s2)+(s1)$) node [above right] {};
        
                \draw [thick,-latex, color=red!60] ($(-2,-3)+(s2)+(s1)+(s2)+(s2)+(s1)+(s1)+(s1)+(s2)+(s1)+(s2)+(s2)+(s1)$)
        -- ($(-2,-3)+(s2)+(s1)+(s2)+(s2)+(s1)+(s1)+(s1)+(s2)+(s1)+(s2)+(s2)+(s1)+(s2)$) node [above right] {};
        
                 \draw [thick,-latex, color=red!60] ($(-2,-3)+(s2)+(s1)+(s2)+(s2)+(s1)+(s1)+(s1)+(s2)+(s1)+(s2)+(s2)+(s1)+(s2)$)
        -- ($(-2,-3)+(s2)+(s1)+(s2)+(s2)+(s1)+(s1)+(s1)+(s2)+(s1)+(s2)+(s2)+(s1)+(s2)+(s1)$) node [above right] {};
        
        \draw [thick,-latex, color=red!60] ($(-2,-3)+(s2)+(s1)+(s2)+(s2)+(s1)+(s1)+(s1)+(s2)+(s1)+(s2)+(s2)+(s1)+(s2)+(s1)$)
        -- ($(-2,-3)+(s2)+(s1)+(s2)+(s2)+(s1)+(s1)+(s1)+(s2)+(s1)+(s2)+(s2)+(s1)+(s2)+(s1)+(s2)$) node [above right] {};
        
           \draw [thick,-latex, color=red!60] ($(-2,-3)+(s2)+(s1)+(s2)+(s2)+(s1)+(s1)+(s1)+(s2)+(s1)+(s2)+(s2)+(s1)+(s2)+(s1)+(s2)$)
        -- ($(-2,-3)+(s2)+(s1)+(s2)+(s2)+(s1)+(s1)+(s1)+(s2)+(s1)+(s2)+(s2)+(s1)+(s2)+(s1)+(s2)+(s2)$) node [above right] {};
        
            \draw [thick,-latex, color=red!60] ($((-2,-3)+(s2)+(s1)+(s2)+(s2)+(s1)+(s1)+(s1)+(s2)+(s1)+(s2)+(s2)+(s1)+(s2)+(s1)+(s2)+(s2)$)
        -- ($(-2,-3)+(s2)+(s1)+(s2)+(s2)+(s1)+(s1)+(s1)+(s2)+(s1)+(s2)+(s2)+(s1)+(s2)+(s1)+(s2)+(s2)+(s1)$) node [above right] {};
   
     \draw [thick,-latex, color=red!60] ($((-2,-3)+(s2)+(s1)+(s2)+(s2)+(s1)+(s1)+(s1)+(s2)+(s1)+(s2)+(s2)+(s1)+(s2)+(s1)+(s2)+(s2)+(s1)$)
        -- ($(-2,-3)+(s2)+(s1)+(s2)+(s2)+(s1)+(s1)+(s1)+(s2)+(s1)+(s2)+(s2)+(s1)+(s2)+(s1)+(s2)+(s2)+(s1)+(s1)$) node [above right] {};
        
         \draw [thick,-latex, color=red!60] ($(-2,-3)+(s2)+(s1)+(s2)+(s2)+(s1)+(s1)+(s1)+(s2)+(s1)+(s2)+(s2)+(s1)+(s2)+(s1)+(s2)+(s2)+(s1)+(s1)$)
        -- ($(-2,-3)+(s2)+(s1)+(s2)+(s2)+(s1)+(s1)+(s1)+(s2)+(s1)+(s2)+(s2)+(s1)+(s2)+(s1)+(s2)+(s2)+(s1)+(s1)+(s1)$) node [above right] {};

  \end{tikzpicture}
  \end{center}
  \caption{A lattice walk in $\mathbb{Z}^2$ that starts at $(0,0)$, consists of steps $(1,1)$ or $(1,-1)$, and crosses the $x$-axis, and its reflection along the line $y = -1$.}
 \end{figure}
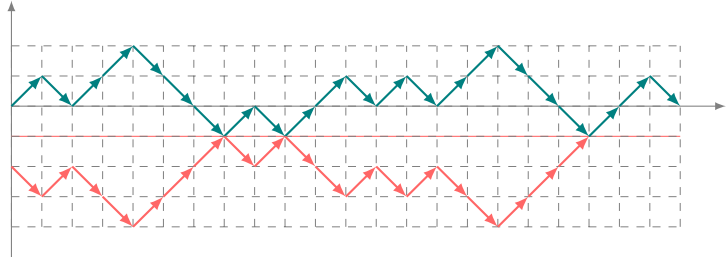

\bigskip

This combinatorial argument is referred to as the \emph{reflection principle}, and the variant of the kernel method which leads to the expression of $F(x;t)$ as the non-negative part of a rational function is therefore sometimes referred to as the algebraic variant of the reflection principle.
\bigskip

The reflection principle also allows one to directly determine the number $f(i;n)$ of walks in $\mathbb{N}$ that start at $0$, end in $i$ and have length $n$. It is 
\begin{equation*}
f(i;n) = \binom{n}{\frac{n-i}{2}} - \binom{n}{\frac{n-i-2}{2}}= \frac{n-i}{n+i+2}\binom{n}{\frac{n-i}{2}},
\end{equation*}
since the binomial coefficient $\binom{n}{\frac{n-i}{2}}$ counts the number of walks in $\mathbb{Z}$ that start at~$0$, end at $i$ and have length $n$.

\section{Combinatorial factorization}\label{sec:combFact}

In Section~\ref{sec:wienerHopf} we met a factorization of the kernel polynomial $K(x,t)$ to which we now give a combinatorial meaning. Let us slightly rewrite it in the form 
\begin{equation*}
K(x,t) = x (1-\bar{x}x_0) t x_1 (1- x x_1^{-1}),
\end{equation*}
and point out that $K(x,t)$ is the product of $x$ and series $s_-$, $s_0$ and $s_+$ such that
\begin{equation*}
s_-^{\pm 1} \in \mathbb{Q}[\bar{x}][[t]], \quad s_+^{\pm 1} \in \mathbb{Q}[x][[t]], \quad \text{and} \quad s_0^{\pm 1}\in\mathbb{Q}[[t]].
\end{equation*}
Furthermore, 
\begin{equation*}
[x^0] s_- = 1 =  [x^0] s_+.
\end{equation*}
The latter makes this decomposition unique. 
%This has the following consequence. 

Any walk on $\mathbb{Z}$ that starts at $0$ and takes its steps from $\{-1,1\}$ can be written as the concatenation of three walks $w_-$, $w_0$ and $w_+$ that have the following property: $w_-$ is a walk that starts at $0$ and does not end at a positive integer, $w_0$ is a walk in $\mathbb{N}$ that starts and ends at $0$, and $w_+$ is a walk in $\mathbb{N}$ that starts at $0$. Requiring that $w_0$ is of maximal length makes this decomposition unique.

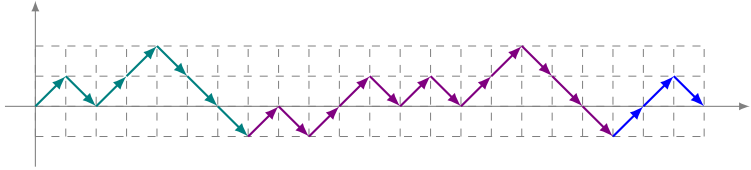
\begin{figure}[ht]
\hspace{-1.5cm}
  \centering
  \begin{center}
  \begin{tikzpicture}[scale=.4]
   
    \coordinate (Origin)   at (-2,-1);
    \coordinate (XAxisMin) at (-3,-1);
    \coordinate (XAxisMax) at (21.5,-1);
    \coordinate (YAxisMin) at (-2,-3);
    \coordinate (YAxisMax) at (-2,2.5);
    \draw [thin, gray,-latex] (XAxisMin) -- (XAxisMax);% Draw x axis
    \draw [thin, gray,-latex] (YAxisMin) -- (YAxisMax);% Draw y axis
 
    \coordinate (s1) at (1,1);
    \coordinate (s2) at (1,-1);

    \draw[style=help lines,dashed] (-2,-2) grid[step=1cm] (20,1);

    %draws the steps of the walk

    \draw [thick,-latex,teal] (Origin)
        -- ($(Origin)+(s1)$) ;

    \draw [thick,-latex,teal] ($(Origin)+(s1)$)
        -- ($(Origin)+(s1)+(s2)$) node [above right] {};

    \draw [thick,-latex,teal] ($(Origin)+(s1)+(s2)$)
        -- ($(Origin)+(s1)+(s2)+(s1)$) node [above right] {};

    \draw [thick,-latex,teal] ($(Origin)+(s1)+(s2)+(s1)$)
        -- ($(Origin)+(s1)+(s2)+(s1)+(s1)$) node [above right] {};

    \draw [thick,-latex,teal] ($(Origin)+(s1)+(s2)+(s1)+(s1)$)
        -- ($(Origin)+(s1)+(s2)+(s1)+(s1)+(s2)$) node [above right] {};

    \draw [thick,-latex,teal] ($(Origin)+(s1)+(s2)+(s1)+(s1)+(s2)$)
        -- ($(Origin)+(s1)+(s2)+(s1)+(s1)+(s2)+(s2)$) node [above right] {};
        
        \draw [thick,-latex,teal] ($(Origin)+(s1)+(s2)+(s1)+(s1)+(s2)+(s2)$)
        -- ($(Origin)+(s1)+(s2)+(s1)+(s1)+(s2)+(s2)+(s2)$) node [above right] {};
        
           \draw [thick,-latex,violet] ($(Origin)+(s1)+(s2)+(s1)+(s1)+(s2)+(s2)+(s2)$)
        -- ($(Origin)+(s1)+(s2)+(s1)+(s1)+(s2)+(s2)+(s2)+(s1)$) node [above right] {};
        
           \draw [thick,-latex,violet] ($(Origin)+(s1)+(s2)+(s1)+(s1)+(s2)+(s2)+(s2)+(s1)$)
        -- ($(Origin)+(s1)+(s2)+(s1)+(s1)+(s2)+(s2)+(s2)+(s1)+(s2)$) node [above right] {};
        
           \draw [thick,-latex,violet] ($(Origin)+(s1)+(s2)+(s1)+(s1)+(s2)+(s2)+(s2)+(s1)+(s2)$)
        -- ($(Origin)+(s1)+(s2)+(s1)+(s1)+(s2)+(s2)+(s2)+(s1)+(s2)+(s1)$) node [above right] {};
        
               \draw [thick,-latex,violet] ($(Origin)+(s1)+(s2)+(s1)+(s1)+(s2)+(s2)+(s2)+(s1)+(s2)+(s1)$)
        -- ($(Origin)+(s1)+(s2)+(s1)+(s1)+(s2)+(s2)+(s2)+(s1)+(s2)+(s1)+(s1)$) node [above right] {};

               \draw [thick,-latex,violet] ($(Origin)+(s1)+(s2)+(s1)+(s1)+(s2)+(s2)+(s2)+(s1)+(s2)+(s1)+(s1)$)
        -- ($(Origin)+(s1)+(s2)+(s1)+(s1)+(s2)+(s2)+(s2)+(s1)+(s2)+(s1)+(s1)+(s2)$) node [above right] {};
        
                \draw [thick,-latex,violet] ($(Origin)+(s1)+(s2)+(s1)+(s1)+(s2)+(s2)+(s2)+(s1)+(s2)+(s1)+(s1)+(s2)$)
        -- ($(Origin)+(s1)+(s2)+(s1)+(s1)+(s2)+(s2)+(s2)+(s1)+(s2)+(s1)+(s1)+(s2)+(s1)$) node [above right] {};
        
                 \draw [thick,-latex,violet] ($(Origin)+(s1)+(s2)+(s1)+(s1)+(s2)+(s2)+(s2)+(s1)+(s2)+(s1)+(s1)+(s2)+(s1)$)
        -- ($(Origin)+(s1)+(s2)+(s1)+(s1)+(s2)+(s2)+(s2)+(s1)+(s2)+(s1)+(s1)+(s2)+(s1)+(s2)$) node [above right] {};
        
        \draw [thick,-latex,violet] ($(Origin)+(s1)+(s2)+(s1)+(s1)+(s2)+(s2)+(s2)+(s1)+(s2)+(s1)+(s1)+(s2)+(s1)+(s2)$)
        -- ($(Origin)+(s1)+(s2)+(s1)+(s1)+(s2)+(s2)+(s2)+(s1)+(s2)+(s1)+(s1)+(s2)+(s1)+(s2)+(s1)$) node [above right] {};
        
           \draw [thick,-latex,violet] ($(Origin)+(s1)+(s2)+(s1)+(s1)+(s2)+(s2)+(s2)+(s1)+(s2)+(s1)+(s1)+(s2)+(s1)+(s2)+(s1)$)
        -- ($(Origin)+(s1)+(s2)+(s1)+(s1)+(s2)+(s2)+(s2)+(s1)+(s2)+(s1)+(s1)+(s2)+(s1)+(s2)+(s1)+(s1)$) node [above right] {};
        
            \draw [thick,-latex,violet] ($(Origin)+(s1)+(s2)+(s1)+(s1)+(s2)+(s2)+(s2)+(s1)+(s2)+(s1)+(s1)+(s2)+(s1)+(s2)+(s1)+(s1)$)
        -- ($(Origin)+(s1)+(s2)+(s1)+(s1)+(s2)+(s2)+(s2)+(s1)+(s2)+(s1)+(s1)+(s2)+(s1)+(s2)+(s1)+(s1)+(s2)$) node [above right] {};
   
     \draw [thick,-latex,violet] ($(Origin)+(s1)+(s2)+(s1)+(s1)+(s2)+(s2)+(s2)+(s1)+(s2)+(s1)+(s1)+(s2)+(s1)+(s2)+(s1)+(s1)+(s2)$)
        -- ($(Origin)+(s1)+(s2)+(s1)+(s1)+(s2)+(s2)+(s2)+(s1)+(s2)+(s1)+(s1)+(s2)+(s1)+(s2)+(s1)+(s1)+(s2)+(s2)$) node [above right] {};
        
         \draw [thick,-latex,violet] ($(Origin)+(s1)+(s2)+(s1)+(s1)+(s2)+(s2)+(s2)+(s1)+(s2)+(s1)+(s1)+(s2)+(s1)+(s2)+(s1)+(s1)+(s2)+(s2)$)
        -- ($(Origin)+(s1)+(s2)+(s1)+(s1)+(s2)+(s2)+(s2)+(s1)+(s2)+(s1)+(s1)+(s2)+(s1)+(s2)+(s1)+(s1)+(s2)+(s2)+(s2)$) node [above right] {};
        
             \draw [thick,-latex,blue] ($(Origin)+(s1)+(s2)+(s1)+(s1)+(s2)+(s2)+(s2)+(s1)+(s2)+(s1)+(s1)+(s2)+(s1)+(s2)+(s1)+(s1)+(s2)+(s2)+(s2)$)
        -- ($(Origin)+(s1)+(s2)+(s1)+(s1)+(s2)+(s2)+(s2)+(s1)+(s2)+(s1)+(s1)+(s2)+(s1)+(s2)+(s1)+(s1)+(s2)+(s2)+(s2)+(s1)$) node [above right] {};
        
         \draw [thick,-latex,blue] ($(Origin)+(s1)+(s2)+(s1)+(s1)+(s2)+(s2)+(s2)+(s1)+(s2)+(s1)+(s1)+(s2)+(s1)+(s2)+(s1)+(s1)+(s2)+(s2)+(s2)+(s1)$)
        -- ($(Origin)+(s1)+(s2)+(s1)+(s1)+(s2)+(s2)+(s2)+(s1)+(s2)+(s1)+(s1)+(s2)+(s1)+(s2)+(s1)+(s1)+(s2)+(s2)+(s2)+(s1)+(s1)$) node [above right] {};
        
         \draw [thick,-latex,blue] ($(Origin)+(s1)+(s2)+(s1)+(s1)+(s2)+(s2)+(s2)+(s1)+(s2)+(s1)+(s1)+(s2)+(s1)+(s2)+(s1)+(s1)+(s2)+(s2)+(s2)+(s1)+(s1)$)
        -- ($(Origin)+(s1)+(s2)+(s1)+(s1)+(s2)+(s2)+(s2)+(s1)+(s2)+(s1)+(s1)+(s2)+(s1)+(s2)+(s1)+(s1)+(s2)+(s2)+(s2)+(s1)+(s1)+(s2)$) node [above right] {};

  \end{tikzpicture}
  \end{center}
  \caption{A lattice walk and its combinatorial factorization.}
 \end{figure}

This translates into a factorization of the generating function of walks on $\mathbb{Z}$
\begin{equation*}
\frac{1}{1- t (x + \bar{x})} = F_-(x;t) F(0;t) F_+(x;t),
\end{equation*}
where again 
\begin{equation*}
F_-^{\pm 1}\in\mathbb{Q}[\bar{x}][[t]], \quad F_+^{\pm 1}\in\mathbb{Q}[\bar{x}][[t]] \quad \text{and} \quad F_0\in\mathbb{Q}[[t]]
\end{equation*}
and 
\begin{equation*}
[x^0] F_- = 1 = [x^0] F_+.
\end{equation*}
It now follows that
\begin{equation*}
F_-(x;t) = \frac{1}{1 - \bar{x}x_0}, \quad F(0;t) = \frac{1}{tx_1} \quad \text{and} \quad F_+(x;t) = \frac{1}{1 - xx_1^{-1}}.
\end{equation*}
Furthermore, since a walk $w$ on $\mathbb{Z}$ is a walk on $\mathbb{N}$ if and only if $w_-$ is of length zero, we find that 
\begin{equation*}
F(x;t) = F(0;t) F_+(x;t) = \frac{1}{t(x_1 - x)}.
\end{equation*}

%Taking into account that 
%\begin{equation*}
%- t + x - t x^2 = -t(x-x_0) (x-x_1)
%\end{equation*}
%and hence 
%\begin{equation*}
%x_0 x_1 = 1 \quad \text{and} \quad x_0 + x_1 = \frac{1}{t},
%\end{equation*}
%it is easily seen that 
%\begin{equation*}
%K(x;t) = 
%\end{equation*}

\section{Cycle lemma}\label{sec:cycleLemma}
We have seen how Lagrange inversion implies that 
\begin{equation*}
f(0;n) = \frac{1}{n+1} [t^{-1}] (t + t^{-1})^{n+1}.
\end{equation*}
Since~$[t^{-1}] (t + t^{-1})^{n+1}$ is the number of walks in $\mathbb{Z}$ which start at $0$, consist of $n+1$ steps all of which are taken from~$\{1,-1\}$ and end in $-1$, this identity begs for a combinatorial explanation. The cycle lemma provides such an explanation. There are different versions of the cycle lemma, the one we present here is due to Spitzer and also known as Spitzer's lemma. For details and proofs we refer to~\cite{Krattenthaler} from which the following lemma and proposition are taken.
\begin{Lemma}
Let $a_1,\dots,a_N$ be a sequence of real numbers such that $a_1+\dots+a_N = 0$ and no partial sum of consecutive $a_i$'s read cyclically vanishes. Then there is a unique cyclic permutation $a_i,a_{i+1},\dots,a_N,a_1,\dots,a_{i-1}$ with the property that for $j=1,\dots,N$ the sum of its first $j$ terms is non-negative.
\end{Lemma}
A consequence of the cycle lemma is the following proposition.
\begin{Proposition}\label{prop:applyCycle}
Let $r$ and $s$ be two positive integers which are relative prime. Then the number of paths in $\mathbb{N}^2$ that start at $(0,0)$, consist of steps $(1,0)$ and $(0,1)$, end at $(r,s)$ and stay weakly below the line $ry = sx$ is $\frac{1}{r+s}\binom{r+s}{s}$.
\end{Proposition}
\begin{proof}
Let $\mathfrak{P}$ be the set of lattice walks that start at $(0,0)$ and end at $(r,s)$ and let~$\mathfrak{G}$ be the group generated by the permutation $(1,2,\dots, r+s)$. Clearly, $\mathfrak{G}$ acts on~$\mathfrak{P}$ by permuting the steps of a walk cyclically. Since $\mathfrak{P}$ has cardinality $\binom{r+s}{s}$ and since each orbit has cardinality $r+s$ the set $\mathfrak{P}$ partitions into $\frac{1}{r+s}\binom{r+s}{s}$ orbits. We finish the proof of the proposition by observing that each orbit contains exactly one walk that lies weakly below the line $ry=sx$. To see this note that we can encode a walk by a sequence of numbers by replacing a step $(1,0)$ by $s$ and a step~$(0,1)$ by~$-r$. Since $r$ and $s$ are relatively prime the cycle lemma applies.
\end{proof}

Proposition~\ref{prop:applyCycle} provides a combinatorial explanation for equation~\eqref{eq:lagrange1} as there is a bijection between walks of length $2n$ in $\mathbb{N}$ that start at~$0$, take their steps from $\{-1,1\}$ and end at $0$ and walks of length $2n+1$ in $\mathbb{N}^2$ that start at $(0,0)$, consist of steps $(1,0)$ or $(0,1)$, end at $(r,s)=(n,n+1)$ and stay weakly below the line $ry = sx$.

\begin{figure}[ht]
\hspace{-1.5cm}
  \centering
  \begin{center}
  \begin{tikzpicture}[scale=.4]
   
    \coordinate (Origin)   at (-2,-1);
    \coordinate (XAxisMin) at (-3,-1);
    \coordinate (XAxisMax) at (21.5,-1);
    \coordinate (YAxisMin) at (-2,-2);
    \coordinate (YAxisMax) at (-2,3.5);
    \draw [thin, gray,-latex] (XAxisMin) -- (XAxisMax);% Draw x axis
    \draw [thin, gray,-latex] (YAxisMin) -- (YAxisMax);% Draw y axis
 
    \coordinate (s1) at (1,1);
    \coordinate (s2) at (1,-1);

    \draw[style=help lines,dashed] (-2,-1) grid[step=1cm] (20,2);

    %draws the steps of the walk

%    \draw [thick,-latex,teal] (Origin)
%        -- ($(Origin)+(s1)$) ;
%
%    \draw [thick,-latex,teal] ($(Origin)+(s1)$)
%        -- ($(Origin)+(s1)+(s2)$) node [above right] {};
%
%    \draw [thick,-latex,teal] ($(Origin)+(s1)+(s2)$)
%        -- ($(Origin)+(s1)+(s2)+(s1)$) node [above right] {};
%
%    \draw [thick,-latex,teal] ($(Origin)+(s1)+(s2)+(s1)$)
%        -- ($(Origin)+(s1)+(s2)+(s1)+(s1)$) node [above right] {};
%
%    \draw [thick,-latex,teal] ($(Origin)+(s1)+(s2)+(s1)+(s1)$)
%        -- ($(Origin)+(s1)+(s2)+(s1)+(s1)+(s2)$) node [above right] {};
%
%    \draw [thick,-latex,teal] ($(Origin)+(s1)+(s2)+(s1)+(s1)+(s2)$)
%        -- ($(Origin)+(s1)+(s2)+(s1)+(s1)+(s2)+(s2)$) node [above right] {};
%        
%        \draw [thick,-latex,teal] ($(Origin)+(s1)+(s2)+(s1)+(s1)+(s2)+(s2)$)
%        -- ($(Origin)+(s1)+(s2)+(s1)+(s1)+(s2)+(s2)+(s2)$) node [above right] {};
        
           \draw [thick,-latex,violet] ($(Origin)$)
        -- ($(Origin)+(s1)$) node [above right] {};
        
         \draw [thick,-latex,violet] ($(Origin)+(s1)$)
        -- ($(Origin)+(s1)+(s2)$) node [above right] {};
        
           \draw [thick,-latex,violet] ($(Origin)+(s1)+(s2)$)
        -- ($(Origin)+(s1)+(s2)+(s1)$) node [above right] {};
        
               \draw [thick,-latex,violet] ($(Origin)+(s1)+(s2)+(s1)$)
        -- ($(Origin)+(s1)+(s2)+(s1)+(s1)$) node [above right] {};

               \draw [thick,-latex,violet] ($(Origin)+(s1)+(s2)+(s1)+(s1)$)
        -- ($(Origin)+(s1)+(s2)+(s1)+(s1)+(s2)$) node [above right] {};
        
                \draw [thick,-latex,violet] ($(Origin)+(s1)+(s2)+(s1)+(s1)+(s2)$)
        -- ($(Origin)+(s1)+(s2)+(s1)+(s1)+(s2)+(s1)$) node [above right] {};
        
                 \draw [thick,-latex,violet] ($(Origin)+(s1)+(s2)+(s1)+(s1)+(s2)+(s1)$)
        -- ($(Origin)+(s1)+(s2)+(s1)+(s1)+(s2)+(s1)+(s2)$) node [above right] {};
        
        \draw [thick,-latex,violet] ($(Origin)+(s1)+(s2)+(s1)+(s1)+(s2)+(s1)+(s2)$)
        -- ($(Origin)+(s1)+(s2)+(s1)+(s1)+(s2)+(s1)+(s2)+(s1)$) node [above right] {};
        
           \draw [thick,-latex,violet] ($(Origin)+(s1)+(s2)+(s1)+(s1)+(s2)+(s1)+(s2)+(s1)$)
        -- ($(Origin)+(s1)+(s2)+(s1)+(s1)+(s2)+(s1)+(s2)+(s1)+(s1)$) node [above right] {};
        
            \draw [thick,-latex,violet] ($(Origin)+(s1)+(s2)+(s1)+(s1)+(s2)+(s1)+(s2)+(s1)+(s1)$)
        -- ($(Origin)+(s1)+(s2)+(s1)+(s1)+(s2)+(s1)+(s2)+(s1)+(s1)+(s2)$) node [above right] {};
   
     \draw [thick,-latex,violet] ($(Origin)+(s1)+(s2)+(s1)+(s1)+(s2)+(s1)+(s2)+(s1)+(s1)+(s2)$)
        -- ($(Origin)+(s1)+(s2)+(s1)+(s1)+(s2)+(s1)+(s2)+(s1)+(s1)+(s2)+(s2)$) node [above right] {};
        
         \draw [thick,-latex,violet] ($(Origin)+(s1)+(s2)+(s1)+(s1)+(s2)+(s1)+(s2)+(s1)+(s1)+(s2)+(s2)$)
        -- ($(Origin)+(s1)+(s2)+(s1)+(s1)+(s2)+(s1)+(s2)+(s1)+(s1)+(s2)+(s2)+(s2)$) node [above right] {};
        
             \draw [thick,-latex,blue] ($(Origin)+(s1)+(s2)+(s1)+(s1)+(s2)+(s1)+(s2)+(s1)+(s1)+(s2)+(s2)+(s2)$)
        -- ($(Origin)+(s1)+(s2)+(s1)+(s1)+(s2)+(s1)+(s2)+(s1)+(s1)+(s2)+(s2)+(s2)+(s1)$) node [above right] {};
        
         \draw [thick,-latex,blue] ($(Origin)+(s1)+(s2)+(s1)+(s1)+(s2)+(s1)+(s2)+(s1)+(s1)+(s2)+(s2)+(s2)+(s1)$)
        -- ($(Origin)+(s1)+(s2)+(s1)+(s1)+(s2)+(s1)+(s2)+(s1)+(s1)+(s2)+(s2)+(s2)+(s1)+(s1)$) node [above right] {};
        
         \draw [thick,-latex,blue] ($(Origin)+(s1)+(s2)+(s1)+(s1)+(s2)+(s1)+(s2)+(s1)+(s1)+(s2)+(s2)+(s2)+(s1)+(s1)$)
        -- ($(Origin)+(s1)+(s2)+(s1)+(s1)+(s2)+(s1)+(s2)+(s1)+(s1)+(s2)+(s2)+(s2)+(s1)+(s1)+(s2)$) node [above right] {};

 \draw [thick,-latex,teal] ($(Origin)+(s1)+(s2)+(s1)+(s1)+(s2)+(s1)+(s2)+(s1)+(s1)+(s2)+(s2)+(s2)+(s1)+(s1)+(s2)$)
        -- ($(Origin)+(s1)+(s2)+(s1)+(s1)+(s2)+(s1)+(s2)+(s1)+(s1)+(s2)+(s2)+(s2)+(s1)+(s1)+(s2)+(s1)$) node [above right] {};
        
         \draw [thick,-latex,teal] ($(Origin)+(s1)+(s2)+(s1)+(s1)+(s2)+(s1)+(s2)+(s1)+(s1)+(s2)+(s2)+(s2)+(s1)+(s1)+(s2)+(s1)$)
        -- ($(Origin)+(s1)+(s2)+(s1)+(s1)+(s2)+(s1)+(s2)+(s1)+(s1)+(s2)+(s2)+(s2)+(s1)+(s1)+(s2)+(s1)+(s2)$) node [above right] {};

    \draw [thick,-latex,teal] ($(Origin)+(s1)+(s2)+(s1)+(s1)+(s2)+(s1)+(s2)+(s1)+(s1)+(s2)+(s2)+(s2)+(s1)+(s1)+(s2)+(s1)+(s2)$)
        -- ($(Origin)+(s1)+(s2)+(s1)+(s1)+(s2)+(s1)+(s2)+(s1)+(s1)+(s2)+(s2)+(s2)+(s1)+(s1)+(s2)+(s1)+(s2)+(s1)$) node [above right] {};
        
          \draw [thick,-latex,teal] ($(Origin)+(s1)+(s2)+(s1)+(s1)+(s2)+(s1)+(s2)+(s1)+(s1)+(s2)+(s2)+(s2)+(s1)+(s1)+(s2)+(s1)+(s2)+(s1)$)
        -- ($(Origin)+(s1)+(s2)+(s1)+(s1)+(s2)+(s1)+(s2)+(s1)+(s1)+(s2)+(s2)+(s2)+(s1)+(s1)+(s2)+(s1)+(s2)+(s1)+(s1)$) node [above right] {};
        
            \draw [thick,-latex,teal] ($(Origin)+(s1)+(s2)+(s1)+(s1)+(s2)+(s1)+(s2)+(s1)+(s1)+(s2)+(s2)+(s2)+(s1)+(s1)+(s2)+(s1)+(s2)+(s1)+(s1)$)
        -- ($(Origin)+(s1)+(s2)+(s1)+(s1)+(s2)+(s1)+(s2)+(s1)+(s1)+(s2)+(s2)+(s2)+(s1)+(s1)+(s2)+(s1)+(s2)+(s1)+(s1)+(s2)$) node [above right] {};

    \draw [thick,-latex,teal] ($(Origin)+(s1)+(s2)+(s1)+(s1)+(s2)+(s1)+(s2)+(s1)+(s1)+(s2)+(s2)+(s2)+(s1)+(s1)+(s2)+(s1)+(s2)+(s1)+(s1)+(s2)$)
        -- ($(Origin)+(s1)+(s2)+(s1)+(s1)+(s2)+(s1)+(s2)+(s1)+(s1)+(s2)+(s2)+(s2)+(s1)+(s1)+(s2)+(s1)+(s2)+(s1)+(s1)+(s2)+(s2)$) node [above right] {};
        
         \draw [thick,-latex,teal] ($(Origin)+(s1)+(s2)+(s1)+(s1)+(s2)+(s1)+(s2)+(s1)+(s1)+(s2)+(s2)+(s2)+(s1)+(s1)+(s2)+(s1)+(s2)+(s1)+(s1)+(s2) +(s2)$)
        -- ($(Origin)+(s1)+(s2)+(s1)+(s1)+(s2)+(s1)+(s2)+(s1)+(s1)+(s2)+(s2)+(s2)+(s1)+(s1)+(s2)+(s1)+(s2)+(s1)+(s1)+(s2)+(s2) + (s2)$) node [above right] {};

  \end{tikzpicture}
  \end{center}
  \caption{A cyclic permutation of the lattice walk of Figure~5 not violating the boundary condition.}
 \end{figure}
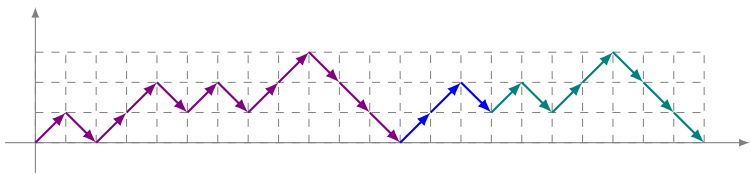

%\bigskip
%
%We saw how useful Lagrange inversion can be because it allowed us to solve equation~\eqref{eq:kEq}. But this is not the only situation in which Lagrange inversion is helpful. The next proposition attributed to Soria, Flajolet, Gessel and Sokal explains how it applies to determine the coefficients of an algebraic series. See~\cite{Drmota} for more information.
%\begin{Proposition}\label{prop:coeffAlg}
%Let $f(x)$ be a power series that is given by its minimal polynomial and its first few terms. Without loss of generality assume that $f(x)$ is defined by $f(x) = P(x,f(x))$ for a known polynomial $P(x,y)$ such that $[y] P \neq 1$ and $P(x,0)\neq 0$. Then 
%\begin{equation*}
%[x^n] f(x) = \sum_{m=0}^n \sum_{
%\substack{m_1+\dots + m_d = m+1\\
%b_1m_1+\dots b_dm_d = n\\
%c_1m_1 + \dots + c_dm_d = m
%}
%} m! \frac{a_1^{m_1}}{m_1!} \dots \frac{a_d^{m_d}}{m_d!},
%\end{equation*}
%where $P(x,y) = \sum_{i=1}^d a_i x^{b_i}y^{c_i}$ and the $m_i$'s are non-negative integers.
%\end{Proposition}
%By applying Proposition~\ref{prop:coeffAlg} to the minimal polynomial of $F(0;t)$ we find the formula for its coefficients we found already several times before. 

\section{Guess and prove}\label{sec:guessProve}

%It is obvious from the expression we found for the generating function $F(x;t)$ of walks in $\mathbb{N}$ that start at $0$ and take their steps from $S = \{-1,1\}$ that $F(x;t)$ and therefore also $F(0;t)$ is algebraic.
Equation~\eqref{eq:returns} shows that the generating function $F(0;t)$ of excursions is algebraic and its minimal polynomial over $Q(t)$ is $P_0(t,Y) = Y^2 - Y/t^2 + 1/t^2$. Since~$F(x;t)$ is a rational function in $F(0;t)$ over $\mathbb{Q}(x,t)$ closure properties of algebraic functions~\cite{theConcreteTetrahedron} imply that $F(x;t)$ is algebraic too. These closure properties are effective, so we could use them to determine the minimal polynomial of $F(x;t)$ given the minimal polynomial of $F(0;t)$. Although in this particular example this would be particularly simple we will proceed differently: we will \emph{guess} the minimal polynomial of $F(x;t)$ and give an argument for its correctness which is independent of the discussion from before. By \emph{guessing} we mean to make an ansatz
\begin{equation*}
P(x,t,Y) = \sum_{i,j,k = 0}^d p_{ijk} x^i t^j Y^k
\end{equation*}
with undetermined coefficients $p_{ijk}$ for (the numerator of) the minimal polynomial of $F(x;t)$. Computation of a truncation of $F(x;t)$ up to some order, its insertion into $P(x,t,Y)$ and setting to zero the coefficients of the monomials in $x$ and $t$ results in a linear system for the $p_{ijk}$'s. In general, there are two situations that can arise: either the linear system has a non-trivial solution which gives rise to an annihilating polynomial for the truncation of $F(x;t)$, or its only solution is zero. If its only solution is zero, then either because the ansatz was not appropriate, that is, the $d$ was too small, or because $F(x;t)$ is not algebraic. For this example we already know that the generating function $F(x;t)$ is algebraic. Indeed, computing $F(x;t)$ up to order $8$ with respect to $t$ and making an ansatz for $P(x,t,Y)$ with $d=2$ we find that the solution space for its coefficients has dimension $1$ and one of its solutions corresponds to
\begin{equation*}
P(x,t,Y) = 1 - (1-2xt) Y - xt(1-t(x + \bar{x})) Y^2.  
\end{equation*}
Of course, having a non-trivial guess for the minimal polynomial does in general not prove the algebraicity of $F(x;t)$. It only means that we have found a polynomial that annihilates $F(x;t)$ up to a given order.
%The reason is that any truncation of $F(x;t)$ is algebraic and by making a suitable ansatz we can always find the minimal polynomial of this truncation. In particular, even if $F(x;t)$ is algebraic, the guessed polynomial need not be its minimal polynomial. 
%On the other hand, if one can not find a reasonable guess for the minimal polynomial of $F(x;t)$ this does not mean that it is not algebraic. It might only mean that the ansatz was not appropriately chosen. 
But having a guess for the minimal polynomial one can increase confidence by checking whether it also annihilates truncations of $F(x;t)$ at higher orders. Having a reasonable candidate for the minimal polynomial, the next step is to verify its correctness. For this particular example it is again particularly easy. Since $P(x,t,Y)$ has degree $2$ with respect to $Y$, we can simply compute its roots and check whether one of them is an element of $\mathbb{Q}[x][[t]]$ that solves equation~\eqref{eq:kEq}. If there is one, then $P(x,t,Y)$ is indeed the minimal polynomial of $F(x;t)$, since $F(x;t)$ is the only solution of the kernel equation in $\mathbb{Q}[x][[t]]$. 
\bigskip

The general situation is more complicated though the underlying principle is the same. The difficulty comes from the fact that the roots of the minimal polynomial may not be given explicitly. But this is actually not necessary. The existence of a root $F_{\text{cand}}(x;t)$ in~$\mathbb{Q}[x][[t]]$ can be derived from the Newton-Puiseux algorithm~\cite{buchacher2025newton}. 
%from the following version of the implicit function theorem~\cite[Theorem~2.9]{theConcreteTetrahedron}.
%\begin{Theorem}
%Let $\mathbb{K}$ be a field of characteristic zero, and let $A(t,y)\in\mathbb{K}[[t,y]]$ be such that 
%\begin{equation*}
%A(0,0) = 0 \quad \text{and} \quad \frac{\partial A}{\partial y} (0,0) \neq 0.
%\end{equation*}
%Then there exists a unique power series $f(t) \in\mathbb{K}[[t]]$ with $f(0) = 0$ and such that~$A(t,f(t)) = 0$.
%\end{Theorem}
%A closer look on how such a root is constructed shows that $F_{\text{cand}}(x;t)$ actually is an element of $\mathbb{Q}[x][[t]]$. 
To prove that $P(x,t,Y)$ is the minimal polynomial of $F(x;t)$ it is sufficient to show that $F(x;t)$ equals $F_{\text{cand}}(x;t)$. Using closure properties of algebraic functions and Gr{\"o}bner bases (or resultants) we find that
\begin{equation*}
A(Y) = Y (1 - 4 t^2 - t^2 Y^2)
\end{equation*}
is an annihilating polynomial for 
\begin{equation}\label{exp:zero}
x(1-t (x + \bar{x})) F_{\text{cand}}(x;t) - x + t F_{\text{cand}}(0;t).
\end{equation}
By computing a truncation of $F_{\text{cand}}(x;t)$ up to some order and plugging the corresponding truncation of expression~\eqref{exp:zero} into the second factor of $A(Y)$ we find that it is not a root of $1 - 4 t^2 - t^2 Y^2$. Since it is a root of $A(Y)$, it is necessarily a root of $Y$, and therefore identically zero. Consequently,~$F_{\text{cand}}(x;t)$ is a solution of the kernel equation and hence equal to~$F(x;t)$ since there is only one solution in~$\mathbb{Q}[x][[t]]$, and~$P(x,t,Y)$ is indeed the minimal polynomial of~$F(x;t)$. 
\bigskip

We now provide some details on the computation of $A(Y)$. By construction $P(x,t,Y)$ is the minimal polynomial of $F_{\text{cand}}(x;t)$, an annihilating polynomial for $F_{\text{cand}}(0;t)$ is $P(0,t,Y)$, and the minimal polynomial of a polynomial $p(x,t)\in\mathbb{Q}[x,t]$ is $p(x,t) - Y$. So we know the annihilating polynomial for the building blocks expression~\eqref{exp:zero} is made of. To find an annihilating polynomial for the full expression we need to effectively perform the closure properties plus and times. This is easily done by using Gr{\"o}bner bases and noting that if $p_1(Y)$ and $p_2(Y)$ are annihilating polynomials for~$y_1$ and~$y_2$ then the generators of the elimination ideals
\begin{equation*}
\langle p_1(y_1), p_2(y_2), y_1 + y_2 - Y \rangle \cap \mathbb{Q}[Y] \quad \text{and} 
\quad \langle p_1(y_1), p_2(y_2), y_1y_2 - Y \rangle \cap \mathbb{Q}[Y]
\end{equation*} 
are annihilating polynomials for $y_1 + y_2$ and $y_1y_2$, respectively. Note that $y_1$ and $y_2$ are treated as variables. 
\bigskip

Having determined the minimal polynomial of $F(x;t)$ it is now natural to ask why it is advantageous to have it at all. In general there is no closed form for its roots and hence for the generating function. But the minimal polynomial together with the first terms of $F(x;t)$, sufficiently many to distinguish it from the other roots, allows one to represent $F(x;t)$ by a finite amount of data. The class of algebraic functions has many closure properties and we have seen before how some of them can be performed effectively. Apart from doing exact computations these representations also admit many other useful operations such as the fast computation of their coefficients, the derivation of closed form expressions, the computation of their asymptotics, and numerical evaluation~\cite{Salvy}. Before we elaborate on the computation of coefficients we introduce the notion of D-finiteness.

\section{Differential algebra}\label{sec:diffAlg}
\begin{Definition}\label{def:dFinite}
A series $f\in\mathbb{Q}[[t_1,\dots,t_n]]$ is called \emph{D-finite}, if its derivatives with respect to each of of its variables form a finite-dimensional vector space over $\mathbb{Q}(t_1,\dots,t_n)$, that is, if for each $t\in\{t_1,\dots,t_n\}$ there are $d\in\mathbb{N}$ and $p_0,\dots,p_d \in\mathbb{Q}[t_1,\dots,t_n]$, not all zero, such that 
\begin{equation*}
p_0 \frac{\partial^d}{\partial t^d}f + p_1 \frac{\partial^{d-1}}{\partial t^{d-1}}f + \dots + p_d f = 0.  
\end{equation*}
%Similarly, a multivariate series $f\in\mathbb{Q}[[t_1,\dots,t_n]]$ is called D-finite if its derivatives with respect to $t_i$ form a finite-dimensional vector space over $\mathbb{Q}(t_1,\dots,t_n)$ for~$i=1,\dots,n$.
\end{Definition}
It is straight-forward to see that a series $f(t) = \sum_{n=0}^\infty f_n t^n$ is D-finite if and only if its coefficient sequence $(f_n)$ satisfies a linear recurrence with polynomial coefficients, i.e. if and only if there are $r\in\mathbb{N}$ and polynomials $q_0,\dots,q_r\in\mathbb{Q}[t]$, not all zero, such that 
\begin{equation*}
q_0(n) f_{n+r} + q_1(n) f_{n+r-1} + \dots + q_r(n) f_n = 0 \quad \text{for all } n\in\mathbb{N}.
\end{equation*}
Consequently, a D-finite function can be encoded by a finite amount of data: a recurrence relation for its coefficient sequence and finitely many of its initial terms. 
\bigskip

%In~\cite{theConcreteTetrahedron} it was shown that 
Any algebraic function is D-finite, and we will now explain how to determine a differential equation for $F(x;t)$, and see how it translates into a recurrence relation for its coefficient sequence. There is more than one way of doing so. For instance, we could use that $F(x;t)$ equals~$\frac{1-\bar{x} x_0}{1-t(x+\bar{x})}$, find differential equations for the building blocks of the expression and then use closure properties of D-finite functions to construct a differential equation for $F(x;t)$ from them. We could also use that $F(x;t)$ is the non-negative part of a rational function and construct a differential equation for $F(x;t)$ using linear algebra as suggested in~\cite{lipshitz} or creative telescoping as proposed in~\cite{ChyzakKauers}. Or we guess a differential equation and verify its correctness similarly as we did before for its minimal polynomial. We will do none of that, but explain how to derive a differential equation directly from its minimal polynomial by unrolling the classical proof that an algebraic function is D-finite. For simplicity we show it for $F_0 \equiv F(0;t)$. Note that $\mathbb{Q}(x,t)[F_0]$ is a finite-dimensional vector space whose dimension equals the degree of the minimal polynomial of $F_0$ and that the extended Euclidean algorithm implies that $\mathbb{Q}(x,t)[F_0]$ is also a field which is closed under taking derivatives. Recall that $P_0(t,Y) = 1-Y+t^2Y^2$ is the minimal polynomial of $F_0$.
By differentiating $P_0(t,F_0) = 0$ with respect to $t$ we find that
\begin{equation}\label{eq:diff0}
\frac{\partial P_0}{\partial t}(t,F_0) + \frac{\partial P_0}{\partial Y}(t,F_0) \frac{\partial F_0}{\partial t} = 0.
\end{equation}
Since $\deg_Y \frac{\partial P_0}{\partial Y} < \deg_Y P_0$, and because $P_0$ is irreducible, there are $u,v\in\mathbb{Q}(t)[Y]$ such that 
\begin{equation*}
u \frac{\partial P_0}{\partial Y} + v P_0 = 1, \quad \text{and hence} \quad u(t,F_0)\frac{\partial P_0}{\partial Y}(t,F_0) = 1.
\end{equation*}
Multiplying equation~\eqref{eq:diff0} by $u(t,F_0)$ and rearranging terms gives
\begin{equation}\label{eq:diff1}
\frac{\partial F_0}{\partial t} = - u(t,F_0) \frac{\partial P_0}{\partial t}(t,F_0). % = \frac{1-(1-2xt)F}{K(x,t)}
\end{equation}
Since $\mathbb{Q}(t)[F_0]$ is a $2$-dimensional vector space there is a non-trivial linear relation between $1$, $F_0$ and $\frac{\partial F_0}{\partial t}$ over $\mathbb{Q}(t)$ that gives rise to a non-trivial linear differential relation for $F_0$:
\begin{equation*}
t(1-4t^2) \frac{\partial F_0}{\partial t} + 2(1-2t^2) F_0 = 2.
%K(x,t) F_x + K_x(x,t) F = 1,
\end{equation*}
By differentiating the equation with respect to $t$ we can get rid of the inhomogeneous part:
\begin{equation}\label{eq:diffeq}
t(1-4 t^2) \frac{\partial^2 F_0}{\partial t^2} + (3-16t^2) \frac{\partial F_0}{\partial t} - 8t F_0 = 0.
%K(x,t) F_{xx} + 2 K_x(x,t) F_x + K_{xx}(x,t) F = 0.
\end{equation}
Hence, $F_0$ is D-finite as stated before.
%Note that by construction the differential equations have minimal order. We point out that there are software packages for Mathematica~\cite{Koutschan} and Maple dealing with D-finite functions which perform these computations as well as many other things automatically. 
From the differential equation one can deduce the following recurrence relation 
\begin{equation}\label{eq:rec1}
(4+n) f(0;n+2) - 4(1+n) f(0;n) = 0
\end{equation}
for the coefficients of $F_0$. It has essentially order $1$, so the coefficient sequence of $F_0$ is essentially hypergeometric~\cite{aEqualsB}. 
%We refer to~\cite{aEqualsB} for more information on hypergeometric sequences. 
In case one does not care whether the recurrence for the coefficients of $F_0$ is linear or not one can as well apply $[t^n]$ to~$F_0 = 1 + t^2 F_0^2$ and find that
\begin{equation}\label{eq:rec2}
f(0;n) = 
\begin{cases} 
      1,\quad  \text{if}\quad  n = 0,\\
      \sum_{k=0}^{n-2} f(0;k) f(0;n-2-k) \quad \text{else}.
   \end{cases}
\end{equation}
The latter recurrence has an evident combinatorial meaning which is not surprising since this is the case for the minimal polynomial of $F_0$ it is derived from. Given an excursion $e$ of length $n\geq 2$ we can write it as the concatenation of excursions $e_1$ and $e_2$ such that $e_1$ visits $0$ at its beginning and end only. If $e_1$ has length $k$ then removing its first and last step induces a bijection between excursions of this kind and excursions of length $k-2$. Since the set of excursions of length $n$ partitions according to the first time the starting point is visited again we have $f(0;n) = \sum_{k=2}^n f(0;k-2) f(0;n-k)$. Together with $f(0;0) = 1$ and $f(0;1)=0$ this gives the recurrence in~\eqref{eq:rec2}. The advantage of the linear recurrence over the non-linear one is the number of operations needed to compute $f(0;n)$. The linear recurrence requires $O(n)$ many arithmetic operations, when used in a naive way, while the other requires $O(n^2)$ many. The computation of $f(0;n)$ for fixed and large $n$ can be done even faster when not caring about the previous terms. We refer the interested reader to~\cite{vonZurGathen,brothers,BostanImpoved}.

The reflection principle implies a simple formula for $g(n) := f(0;2n)$. It can also be deduced from the recurrence relation it satisfies:
\begin{equation*}
g(n) = \prod_{k=0}^{n-1} \frac{4(1+2k)}{4+2k} = \frac{4^n \frac{(2n)!}{2^n n!}}{2^n (n+1)!} = \frac{(2n)!}{n! (n+1)!} = \frac{1}{n+1}\binom{2n}{n}.
\end{equation*}
We point out that there are algorithms for finding closed form solutions of linear recurrences such as polynomial and rational, hypergeometric and d'Alembertian solutions. We refer to~\cite{kauers2023d} for more information on that. The differential equation~\eqref{eq:diffeq} we derived for $F_0$ is not only useful for encoding the generating function, performing exact computations, or deriving linear recurrences for its coefficient sequence. It also allows one to decide whether $F_0$ has some closed form representation and to find it in case there is one. The differential operator $L$ that corresponds to the differential equation~\eqref{eq:diffeq} is
\begin{equation*}
L = t(1-4t^2)\frac{\partial^2}{\partial t^2} + (3-16t^2) \frac{\partial}{\partial t} - 8t.
\end{equation*}
It is the least common left multiple of the operators
\begin{equation*}
L_1 =  t \frac{\partial}{\partial t} + 2 \quad \text{and} \quad L_2 = t(2t-1)(2t+1) \frac{\partial}{\partial t} + 2(2t^2-1),
\end{equation*}
and its solution space is the direct sum of the solution spaces of~$L_1$ and~$L_2$.
%~\cite{manuelsBook}. 
The latter are easily determined since~$L_1$ and~$L_2$ both have order $1$. We just note that the solutions of a differential equation of the form $q_0 f' + q_1 f = 0$, with $q_0,q_1\in\mathbb{Q}(t)$, are constant multiples of $\exp(-\int \frac{q_1}{q_0} \mathrm{dt})$. The solution space of $L_1$ is therefore generated by $f_1 = 1/t^2$, and the solution space of $L_2$ is easily seen to be generated by $f_2 = \frac{\sqrt{1-4t^2}}{t^2}$ by performing the partial fraction decomposition
\begin{equation*}
\frac{2(2t^2-1)}{t(2t-1)(2t+1)} = \frac{2}{t} + \frac{1}{1-2t} - \frac{1}{1+2t}
\end{equation*}
and integrating each summand separately. Both, $f_1$ and $f_2$, have a singularity at~$0$, a pole of order $2$. It can be eliminated by linearly combining $f_1$ and $f_2$ to $f := (1-\sqrt{1-4t^2})/t^2$. By comparing $f(0)=2$ 
%and $f'(0) = 0$ 
with $f(0;0)=1$ 
%and $f(0;2)$ 
we find that 
\begin{equation*}
F(0;t) = \frac{1-\sqrt{1-4t^2}}{2 t^2}.
\end{equation*}

\section{Asymptotic methods}\label{sec:asympMethod}
Often one is not interested in an exact formula for a combinatorial sequence but in an \emph{asymptotic} one. There can be different reasons for that. It could be that there is no simple exact formula, or that there is a one, yet it is difficult to evaluate. There are different ways an asymptotic formula can be derived. For instance, using Stirling's formula,
\begin{equation*}
n! = \left( \frac{n}{e} \right)^n \sqrt{2 \pi n} \left( 1 + O\left(\frac{1}{n} \right)\right), \quad n \rightarrow \infty,
\end{equation*}
one can find for 
\begin{equation*}
g(n) = \frac{1}{n+1} \binom{2n}{n}
\end{equation*}
the following asymptotic approximation
\begin{equation}\label{eq:asymptotics}
g(n) \sim \frac{1}{\sqrt{\pi}} 4^n n ^{-\frac{3}{2}}, \quad n \rightarrow \infty.
\end{equation}
This asymptotic formula consists of several parts: $4^n$ and $n^{-\frac{3}{2}}$ are called the \emph{exponential} and \emph{polynomial growth}, respectively, and $1/\sqrt{\pi}$ is referred to as the \emph{growth constant.} 
%We have just derived the asymptotics of a sequence from an exact formula for its terms. 
\bigskip

We next explain how to derive information about the asymptotics of the sequence $(g(n))$ from the linear recurrence it satisfies. It is based on the notion of generalized series solutions.

\begin{Definition}
A \emph{generalized series} is a $\mathbb{C}$-linear combination of formal objects of the form
\begin{equation}\label{eq:genSeries}
\begin{aligned}
\Gamma(x)^{u/v} &\phi^x \exp(s_1 x^{1/v} + s_2 x^{2/v} + \dots + s_{v-1}x^{(v-1)/v})\\
 \times &x^\alpha ( (c_{0,0} + c_{0,1}x^{-1/v} + c_{0,2} x^{-2/v} + \dots)\\
& \quad + (c_{1,0} + c_{1,1}x^{-1/v} + c_{1,2}x^{-2/v} + \dots ) \log(x)\\
& \quad + \dots\\
& \quad + (c_{m,0} + c_{m,1}x^{-1/v} + c_{m,2}x^{-2/v} + \dots  )\log(x)^m),
\end{aligned}
\end{equation}
where $u \in \mathbb{Z}$, $v\in\mathbb{N}\setminus\{0\}$ and $m\in\mathbb{N}$, and $\phi,\ \alpha,\ s_1,\dots, s_{v-1},\ c_{i,j}\in \mathbb{C}$, and its parts behave as follows with respect to the shift operation:
\begin{equation*}
\Gamma(x+i) = x^{\overline{i}} \Gamma(x), \qquad \phi^{x+i} = \phi^i \phi^x, \qquad 
(x+i)^\alpha = \sum_{n=0}^\infty \binom{\alpha}{n}i^\alpha x^{\alpha - n},
\end{equation*}
\begin{equation*}
\exp(s_l(x+i)^{l/v}) 
%&=  \exp(s_l x^{l/v})\exp(s_l(x+i)^{l/v}-x^{l/v})\\
= \exp(s_l x^{l/v}) \sum_{n=0}^\infty \frac{1}{n!} \left( s_l \sum_{k = 1}^\infty \binom{l/v}{k}i^kx^{l/v-k}\right)^n
\end{equation*}
and 
\begin{equation*}
\log(x+i) = 
%\log(x) + \log(1+\frac{i}{x}) = 
\log(x) - \sum_{n=1}^\infty \frac{(-i)^n}{n}x^{-n}.
\end{equation*}
\end{Definition}
A generalized series which solves a given linear recurrence is called a generalized series solution. It can be shown that if the recurrence has order $r$, then there are exactly as many linearly independent generalized series solutions, all of which are of the above form. Their terms can be computed term by term, see~\cite{manuelAsymp} for an implementation of an algorithm in Mathematica. The recurrence for $(g(n))$,
\begin{equation*}
(4+2n)g(n+1) - 4(1+2n)g(n) = 0,
\end{equation*}
has, up to multiples, only one generalized series solution. Its first terms are
\begin{equation*}
4^x x^{-3/2} \left(1 -\frac{9}{8x} + \frac{145}{128 x^2}  -\frac{1155}{1024 x^3} + \frac{36939}{32768 x^4} + \mathrm{O}(x^{-5}) \right).
\end{equation*}
Although a priori only correct in an algebraic sense, any sequence that solves a linear recurrence has an asymptotic expansion which is a linear combination of generalized series solutions~\cite{WimpBerger, Birkhoff, Birkhoff1}. Hence, there is a $c\in \mathbb{C}$ such that 
\begin{equation*}
g(n) = c 4^n n^{-3/2} \left(1 -\frac{9}{8n} + \frac{145}{128 n^2}  -\frac{1155}{1024 n^3} + \frac{36939}{32768 n^4} + \mathrm{O}(n^{-5}) \right), \quad n \rightarrow \infty.
\end{equation*}
The exponential and polynomial growth of $g(n)$ is therefore $4^n$ and $n^{-3/2}$, respectively. In general, it is not known how to determine the coefficients of the linear combination. However, numerical approximations can be found by comparing $g(n)$ with
\begin{equation*}
4^n n^{-3/2} \left(1 -\frac{9}{8n} + \frac{145}{128 n^2}  -\frac{1155}{1024 n^3} + \frac{36939}{32768 n^4} \right)
\end{equation*}
for different and large values of $n$~\cite{kauers2011mathematica}. Based thereon it is often possible to guess their exact values~\cite{kauers2018guess}. 

%TODO: tutte's invariant method, the analytic method, michael singer's approach, symmetric functions, asymptotic counting, analytic combinatorics in several variables, connections to probability theory, harmonic functions, potential theory
%
%TODO: Does the minimal polynomial of $F$ have a combinatorial meaning? What is the advantage of a linear recurrence relation for the coefficients of $F$ over the recurrence for the coefficients of $F$ with respect to $t$ which immediately follows from the functional equation? Analytic kernel method? Continued fractions and orthogonal polynomials? Symmetric functions? Lagrange Inversion? What else can we do with the minimal polynomial? Formulas for the coefficients? Asymptotics? 

\section{Some references}\label{sec:references}

The methods which have been presented here are not new. They have appeared in the work of many other people. In the following section we explain where they can be found.
\bigskip

The method presented in Section~\ref{sec:kernelMethod} dates back at least to~\cite{Knuth}, where it is presented as the solution to Exercise~2.2.1.-4. Having found considerable attention since then~\cite{bousquet2006polynomial,buchacher2018inhomogeneous,notarantonio2023systems,bostan2022algorithms, bostan2023fast}, it is meanwhile referred to as the classical kernel method. 
%A generalization to ordinary discrete differential equations was introduced  in~\cite{bousquet2006polynomial}, and extended further to linear and not necessarily linear systems of such equations in~\cite{buchacher2018inhomogeneous} and~\cite{notarantonio2023systems}, respectively. A discussion of algorithmic aspects is found in~\cite{bostan2022algorithms, bostan2023fast}.

%The idea of solving equation~\eqref{eq:kEq} by reading off the roots of the polynomial on the right-hand side from the kernel polynomial on the left-hand side of the equation appeared in~\cite{Gessel}. TODO: where precisely?

The approach of Section~\ref{sec:wienerHopf} which is based on a particular factorization of the kernel polynomial 
%that transforms the kernel equation into an equation whose left-hand side involves only non-negative powers of $x$ while its right-hand side involves only non-positive powers and eliminates one of the unknowns by applying $[x^\geq]$ 
has not appeared before in this formal / algebraic setting. However, the usefulness of such factorizations is known in the context of Riemann-Hilbert boundary problems~\cite[Section~2.2.3]{Fayolle},~\cite{gohberg2003overview}. 

The orbit-sum method presented in~Section~\ref{sec:orbitSum} was introduced in~\cite{smallSteps}. Its generalization was initiated in~\cite{bostan2021counting} and then continued in~\cite{buchacher2022orbit} based on the Newton-Puiseux algorithm~\cite{buchacher2025newton}.

It is well-known that compositional inverses and Lagrange inversion relate to the enumeration of lattice walks~\cite{Gessel}. Yet, to the best of our knowledge, it has not been applied before as shown in Section~\ref{sec:compInv}.

Inspired by Tutte's work on the enumeration of maps~\cite{tutte1995chromatic}, the invariant method sketched in Section~\ref{sec:invMethod} was introduced and applied to the enumeration of lattice walks in~\cite{bernardi2020counting,bousquet2023enumeration, bousquet2025stationary}. Related algorithmic questions have received considerable attention in~\cite{buchacher2020separating, buchacher2024separating, buchacher2024separated, buchacher2025finite} and~\cite{bonnet2024galoisian}.

The connection of lattice walks with continued fractions and orthogonal polynomials sketched in Section~\ref{sec:contFrac} is classical, and so are the reflection principle and the cycle lemma illustrated in~\ref{sec:reflecPrinciple} and~\ref{sec:cycleLemma}. We refer to~\cite{flajolet1980},~\cite[Sec~10.9 - 10.11]{Krattenthaler} and~\cite{randomWalksInAWeylChamber},~\cite[Sec~10.3, 10.18]{Krattenthaler} and~\cite{spitzer1956combinatorial},~\cite[Sec~10.4]{Krattenthaler}, respectively, and the references therein for further information.

%Similarly classical are the reflection principle and the cycle lemma that were illustrated in~\ref{sec:reflecPrinciple} and~\ref{sec:cycleLemma}. We refer to~\cite{randomWalksInAWeylChamber},~\cite[Sec~10.3, 10.18]{Krattenthaler} and~\cite{spitzer1956combinatorial},~\cite[Sec~10.4]{Krattenthaler}, respectively, and the references therein.

The combinatorial factorization presented in Section~\ref{sec:combFact} is taken and adapted from~\cite{Gessel}.

The paradigm of guess and prove illustrated in Section~\ref{sec:guessProve} has ever been popular in science~\cite{polya1957solve}. However, only since the advent of high performing computers it has given rise to powerful effective methods. The availability of nowadays hard- and software makes the class of D-finite functions introduced in Section~\ref{sec:diffAlg} accessible to practical computations. We refer to~\cite{kauers2023d} for extensive information on that.  
\bigskip

There are methods we have not discussed here though they do fall within the framework of this text. We have not presented them because they cannot be applied to equation~\eqref{eq:kEq}. We refer to~\cite{smallSteps, beaton2018exact} and~\cite{bousquet2016elementary} and~\cite{Mishna,Mishna1}
%some methods we have not mentioned, either because the methods break down due to the simplicity of the problem or the arguments do not involve any manipulation of the functional equation. 
for the half-orbit-sum method, the obstinate kernel method and the iterated kernel method, respectively.

\section{Acknowledgements}

Thanks go to the Johannes Kepler University Linz and the state of Upper Austria which supported parts of this work with the grant LIT-2022-11-YOU-214. Thanks also go to Christian Krattenthaler for his comments on an earlier version of the manuscript.

\bibliographystyle{plain}
\bibliography{survey}

\end{document}